\documentclass[11pt,oneside,reqno]{amsart}
\usepackage[margin=1in]{geometry}
\providecommand{\keywords}[1]
{\small	
\usepackage[nopar]{lipsum} \textbf{\textit{Keywords---}} #1
}

\usepackage[toc]{appendix}
\usepackage{booktabs}
\usepackage{tabularx}
\usepackage{ltablex}
\keepXColumns

\usepackage{threeparttable}
\usepackage{array} 

\usepackage{lipsum}

\usepackage{tikz}
\usetikzlibrary{arrows.meta,positioning,calc}
\usepackage{graphicx} 

\usepackage{latexsym}
\usepackage[maxbibnames=99,citestyle=numeric-comp]{biblatex}
\usepackage{amsmath, amsthm, amssymb, latexsym,epsfig,amsthm, enumerate,multicol,wasysym}
\usepackage{bm}
\usepackage{mathtools}
\usepackage{color}
\usepackage{graphicx}
\usepackage{nccrules}
\usepackage{xfrac}
\usepackage{hyperref}
\hypersetup{
colorlinks,
citecolor=red,
filecolor=red,
linkcolor=red,
urlcolor=black
}
\usepackage{textcomp}
\usepackage{tikz}
\usepackage[dvipsnames]{xcolor}
\usepackage{comment, todonotes}
\usepackage{enumitem}

\usepackage{emptypage}
\usepackage{placeins}

\numberwithin{equation}{section}
\theoremstyle{plain}
\newtheorem{theorem}{Theorem}[section]
\newtheorem{lemma}[theorem]{Lemma}
\newtheorem{corollary}[theorem]{Corollary}
\newtheorem{proposition}[theorem]{Proposition}

\theoremstyle{definition}
\newtheorem{definition}[theorem]{Definition}

\newtheorem*{corollary*}{Corollary}
\newtheorem*{lemma*}{Lemma}

\newtheorem{case[theorem]}{Case}

\usepackage{caption}
\usepackage{mwe}
\usepackage{wrapfig}

\theoremstyle{remark}
\newtheorem{remark}[theorem]{Remark}

\numberwithin{equation}{section}

\usepackage[normalem]{ulem} 

\providecommand{\keywords}[1]
{
\small	
\textbf{\textit{Keywords---}} #1
}

\makeatletter
\renewcommand\@makefntext[1]{%
\noindent
\makebox[1em][r]{\@makefnmark}#1}
\makeatother 

\usepackage{ulem} 
\usepackage{float}

\newcommand{\Z}{\mathbb Z}
\newcommand{\N}{\mathbb N}
\newcommand{\R}{\mathbb{R}}
\newcommand{\C}{\mathbb C}
\newcommand{\cA}{\mathcal A}

\newcommand{\cB}{\mathcal B}

\newcommand{\cK}{\mathcal K}
\newcommand{\cH}{\mathcal H}

\newcommand{\cM}{\mathcal M}
\newcommand{\cI}{\mathcal I}

\newcommand{\cG}{\mathcal G}

\newcommand{\supp}{\mathrm{supp}}
\newcommand{\dist}{\mathrm{dist}}
\newcommand{\diam}{\mathrm{diam}}
\newcommand{\tr}{\mathrm{tr}}

\makeatletter
\renewcommand{\l@section}{\@tocline{1}{0pt}{2em}{}{}}
\renewcommand{\l@subsection}{\@tocline{2}{0em}{3.2em}{}{}}
\renewcommand{\l@subsubsection}{\@tocline{3}{0em}{4.0em}{}{}}
\makeatother

 \title{Localized frames on Euclidean balls}

\author{kevin Hughes}
\address{Independent Researcher, Edinburgh, UK}
\email{khughes.math@gmail.com}

\author{Arie Israel}
\address{Department of Mathematics, The University of Texas at Austin, Austin, TX, USA}
\email{arie@math.utexas.edu}

\author{Azita Mayeli}
\address{Department of Mathematics, CUNY, The Graduate Center, NY, USA}
\email{amayeli@gc.cuny.edu}

\thanks{A.~ Israel was supported by the Air Force Office of Scientific Research, under award FA9550-19-1-0005 and the National Science Foundation grant DMS-2453770.}

\thanks{A.~Mayeli was supported in part by the National Science Foundation grant DMS-2453769, the AMS-Simons Research Enhancement Grant, and the PSC-CUNY grants 67807-00 56 and 67807-00 57.}

\date{\today}

\begin{document}
\maketitle

\begin{abstract}
We construct explicit wave packet frames adapted to Euclidean balls and use them to obtain quantitative eigenvalue estimates for spatio--spectral limiting operators.
Let \(d\geq 2\), let \(B_d(R)\subset \R^d\) be the Euclidean ball of
radius \(R\), and let \(S\subset \R^d\) be a measurable set such that $\partial S$ has finite $(d-\eta)$-upper Minkowski content for $0 < \eta \leq 1$. 
We construct a unit-norm frame for \(L^2(B_d(R))\), with frame bounds depending only on the dimension $d$, whose elements are adapted to the radial and angular geometry of the ball.  
We prove quantitative Fourier localization estimates for this frame: Relative to \(S\), the frame decomposes into packets concentrated in \(S\), packets concentrated in \(\R^d\setminus S\), and an exceptional family whose cardinality is bounded explicitly in terms of \(R\), and the Minkowski content of \(\partial S\). 
As an application, we derive an upper bound for the plunge region of the spatio-spectral limiting operator 
associated to the sets $B_d(R)$ and $S$.
\end{abstract}

\setcounter{tocdepth}{1}

\section{Introduction}

\subsection{Background and motivation}
A fundamental problem in time--frequency analysis is the construction of systems of functions that are simultaneously well localized in both the space and frequency domains. Given two measurable sets $F, S\subset \R^d$, one seeks systems of functions that are concentrated in $F$ and whose Fourier transforms are concentrated in $S$. Such functions are naturally associated with spatio--spectral limiting operators. Let \(P_F : L^2(\R^d) \rightarrow L^2(\R^d)\) denote the orthogonal projection onto functions supported in $F$, 
\[
P_F f := \mathbf{1}_F f,
\]
and let \(B_S : L^2(\R^d) \rightarrow L^2(\R^d)\) denote the band--limiting projection,
\[
B_S f := \mathcal{F}^{-1}(\mathbf{1}_S \mathcal{F} f).
\]
The associated spatio--spectral limiting operator is \[
T_{F,S}= P_F B_S P_F,
\]
or equivalently \(B_S P_F B_S\), since the two operators have the same nonzero eigenvalues, counted with multiplicity.
If \(f\) is a normalized eigenfunction of \(T_{F,S}\), then its eigenvalue measures the proportion of its Fourier energy contained in \(S\). Thus, eigenvalues near \(1\) correspond to functions that are highly concentrated in both \(F\) and \(S\), while the number of eigenvalues above a fixed threshold measures the effective dimension of the class of such functions.

The eigenfunctions associated with the largest eigenvalues solve the corresponding concentration problem optimally, but they are generally difficult to construct or analyze explicitly. In the classical one--dimensional setting, where both $F$ and $S$ are intervals, the eigenfunctions are the prolate spheroidal wave functions. These functions have a rich mathematical theory, but in applications they are often computed numerically rather than represented by simple explicit formulas \cite{osipov2014evaluation}. In higher dimensions, numerical methods have been developed for the highly symmetric case in which $F$ and $S$ are Euclidean balls, where rotational symmetry reduces the problem to one--dimensional computations \cite{greengard2024}. For further discussion and additional references, see \cite{HughesIsraelMayeli_SampTA2025}. 

These considerations naturally motivate the search for alternative systems of functions that retain many of the localization properties of the eigenfunctions while admitting a more explicit construction. This approach was initiated in \cite{israel15eigenvalue} and developed in our papers for cubes in $\R^d$ \cite{israelmayeli2023acha} and for the disk in $\R^2$ \cite{HIM-disk}; see also \cite{HughesIsraelMayeli_SampTA2025}. In these papers, we introduced wave packet systems as flexible families of functions sharing many of the features of Gabor and wavelet systems. 
The present paper continues this program by constructing a wave packet frame on Euclidean balls in arbitrary dimensions.

\subsection{Main results} 
Our first main result constructs an explicit unit-norm frame $\{\psi_\nu\}_{\nu\in\cI}$ for $L^2(B_d(R))$, with frame bounds independent of $R$. Given a frequency set $S$ and a tolerance $\varepsilon>0$, the index set admits a partition $\cI=\cI_1\cup \cI_2\cup \cI_3$ such that the functions indexed by \( \cI_1\) are almost frequency--localized outside \(S\), those indexed by \(\cI_2\) are almost frequency--localized inside \( S\), and only a small exceptional family indexed by \(\cI_3\) remains. The size of this exceptional family is controlled by the geometric complexity of $\partial S$, measured in terms of its upper Minkowski content.

\begin{definition}[Upper Minkowski content]\label{def:M_content}
Let \(0<\eta\leq 1\), and let \(E \subset \R^d\). Denote the $t$-neighborhood of  $E$ by 
\[
E_t = \{ x \in \R^d : \dist(x,E) \leq t \}
.\] The \((d-\eta)\)-upper Minkowski content of $E$ is the quantity
\[
\cM^{*,d-\eta}(E)
:=
\limsup_{r\downarrow 0}
\frac{
\left|E_r\right|
}{r^\eta},
\]
where $|\cdot|$ denotes $d$-dimensional Lebesgue measure. If $E$ contains at least two points, define the parameter
\[
\cM^{d-\eta}(E) := \sup_{0< r \leq \diam(E)} \frac{\left|E_r \right|}{r^\eta}.
\]
\end{definition}

The quantity $\cM^{d-\eta}(E)$ is homogeneous of degree $d-\eta$ under dilations and controls the size of neighborhoods of $E$ at all scales up to $\diam(E)$. For bounded $E$, it is finite whenever the classical upper Minkowski content is finite.

We fix a parameter $s>1$ which specifies the Gevrey regularity used in the construction of the wave packets, and determines the logarithmic exponents in the resulting estimates.

Suppose that \(S \subset \R^d\) is bounded and its boundary has finite \((d-\eta)\)-upper Minkowski content, and let \(R \geq 1\),  \(\varepsilon \in (0,1)\) satisfy $R \cdot \diam(S) \geq 2$. Define the following geometric error term
\begin{equation}\label{eqn:main_err_term}
H_\eta(S,R,\varepsilon) :=  
\cM^{d-\eta}(\partial S) R^{d-\eta}
\begin{cases}
\big(R\cdot\diam(S)\big)^{\eta-1} \log\left(\frac{R\cdot \diam(S)}{\varepsilon}\right)^{sd+1} 
\\ \qquad\qquad 
+ \log\left(\frac{R\cdot\diam(S)}{\varepsilon}\right)^{s\eta} 
& 0<\eta<1
\\
\log\left(\frac{R\cdot\diam(S)}{\varepsilon}\right)^{sd+1} 
& \eta= 1.
\end{cases}
\end{equation}
We can now state the main theorem.
\begin{theorem}[Energy concentration]\label{thm:ball-frame-energy}
Fix \(d \geq 2\) and $s>1$.  
Given a dyadic number $R\geq 2$, there exists a unit-norm frame $\{\psi_\nu\}_{\nu\in \cI}$ for $L^2(B_d(R))$ with frame bounds $0<A<B<\infty$ depending only on $d$ (and independent of $s$ and $R$). If $\varepsilon\in (0,1/2]$ and \(S\subset \R^d\) is a bounded measurable set satisfying $\cM^{d-\eta}(\partial S) < \infty$ for some \(0<\eta\leq 1\) and \(R \cdot \diam(S) \geq 2\), then there exists a partition of the index set $\cI=\cI_1\cup \cI_2\cup \cI_3$ so that \[ \# \cI_3 \lesssim_{d,s,\eta} H_\eta(S,R,\varepsilon) \] and 
\begin{equation}\label{ineq:energy-estimate}
\sum_{\nu\in \cI_1}\|\widehat{\psi_\nu}\|_{L^2(S)}^2
+
\sum_{\nu\in \cI_2}\|\widehat{\psi_\nu}\|_{L^2(\R^d\setminus S)}^2
\leq \varepsilon^2 
.\end{equation}
\end{theorem}

The restriction to dyadic $R$ in Theorem \ref{thm:ball-frame-energy} is artificial and can be removed by rescaling.

The eigenfunctions of $T_R=P_{B_d(R)}B_SP_{B_d(R)}$ are the optimally concentrated functions associated with the pair $(B_d(R),S)$.
At the level of existence alone, they therefore satisfy a stronger
version of the conclusion of Theorem \ref{thm:ball-frame-energy}, with a sharper error term. 
Their drawback is that they are generally not available in explicit form. Theorem \ref{thm:ball-frame-energy} provides a constructive alternative: a complete system whose elements are almost optimally frequency localized, apart from a controlled exceptional family.

We apply Theorem~\ref{thm:ball-frame-energy}, together with a general
eigenvalue counting lemma for frames in a Hilbert space, to obtain the following bound on the plunge region of spatio-spectral limiting operators.
\begin{theorem}[Plunge region bound]\label{thm:ball-plunge}
Let \(S\subset \R^d\) be a bounded measurable set satisfying $\cM^{d-\eta}(\partial S) < \infty$ for some \(0<\eta\leq 1\).
If $s>1$ and $\varepsilon\in(0,1/2]$, then, for the spatio-spectral limiting operator \(T_R:=P_{B_d(R)} B_S P_{B_d(R)}\), we have the following bounds on its plunge region whenever   $R\cdot \diam(S) \geq 2$: 
\begin{align}\label{bd:plunge}
\#\{k:\lambda_k(T_R)\in(\varepsilon,1-\varepsilon)\} \lesssim_{d,s,\eta} H_\eta(S,R,\varepsilon).
\end{align}
\end{theorem}

As a corollary, we obtain the following eigenvalue counting estimate.

\begin{corollary}\label{cor:Landau-quant-esti} Fix $s>1$. Let \(S\subset \R^d\) be a bounded measurable set satisfying $\cM^{d-\eta}(\partial S) < \infty$ for some \(0<\eta\leq 1\). Suppose $R \cdot \diam(S) \geq 2$. For
$\varepsilon\in(0,1)$ define
\[
N_\varepsilon(R):=\#\{k:\lambda_k(T_R)>\varepsilon\}.
\]
Set $\varepsilon_* = \min \{ \varepsilon, 1 - \varepsilon \} \in (0,1/2]$. Then 
\begin{equation}\label{eig_dist_est}
\left|N_\varepsilon(R) - (2\pi)^{-d}|B_d(R)||S|\right| \lesssim_{d,s,\eta} H_\eta(S,R,\varepsilon_*).
\end{equation}
\end{corollary}

Corollary \ref{cor:Landau-quant-esti} provides a Weyl-type estimate for the number of eigenvalues of \(T_R\) exceeding \(\varepsilon\).

\subsection{Relation to previous work} 
The study of the eigenvalue distribution of spatio--spectral limiting operators goes back to the classical work of Slepian, Pollak, and Landau on time--frequency concentration \cite{Bell1,Bell2}, and to the later work of Landau and Widom \cite{LandauWidom80}. In the one-dimensional setting, this theory gives a spectral form of the \(2WT\) principle: roughly \(2WT\) independent signals can be well concentrated simultaneously in a time interval of length \(T\) and a frequency band of width \(2W\).\footnote{This statement of the $2WT$ principle does not see the factor of $(2\pi)^{-1}$ present in \eqref{eig_dist_est} due to the different normalization of the Fourier transform adopted by the engineering community.}

Our eigenvalue estimates arise from an explicit wave packet construction together with the energy estimate \eqref{ineq:energy-estimate}. A distinguishing feature of this approach is that it yields an explicit family of functions with nearly optimal localization properties. This continues the program developed in \cite{israelmayeli2023acha,HIM-disk}, where analogous constructions were obtained for the cube and for the disk in two dimensions. The present work extends this framework to Euclidean balls in arbitrary dimensions and allows the frequency set to have boundary of fractional codimension \(\eta\in(0,1]\). The principal new ingredient in the proof is the development of a geometric calculus for Gevrey functions on the sphere in local geodesic coordinates. By locally flattening the spherical boundary in suitable coordinate charts, we exploit this Gevrey regularity to obtain the Fourier decay estimates required for the wave packet construction.

An alternative approach to quantitative eigenvalue estimates is based on decomposing the limiting operator and estimating the Schatten--von Neumann norms of its components. This strategy was first developed in this context in \cite{marceca2023}. For the current state of the art, see \cite{Kulikov-Larsen}.

By contrast, the results of \cite{Kulikov-Larsen} apply in a more general geometric setting. Suppose $F$ and $S$ are domains whose boundaries  \(\partial F\) and \(\partial S\) have finite $(d-1)$-upper Minkowski content. Let $RF$ denote the isotropic scaling of $F$ by a factor of $R$. Then the following bound holds on the cardinality of the plunge region for the operator $T_R = P_{RF}B_S P_{RF}$. Provided $\alpha^{-R} < \varepsilon < 1/2$,
\begin{equation}\label{bd:KL26}
\# \{ k : \lambda_k(T_R) \in (\varepsilon,1-\varepsilon) \}
\lesssim \widetilde{H}(R,\varepsilon)
\end{equation} 
with
\begin{equation}\label{eqn:KL_error}
\widetilde{H}(R,\varepsilon) = R^{d-1} \log(\varepsilon^{-1}) \log^2 \left( \alpha R/\log(\varepsilon^{-1}) \right),
\end{equation}
where $\alpha$ is a constant determined by $F$ and $S$. The same paper establishes a different estimate in the regime $\varepsilon \leq \alpha^{-R}$.

A direct comparison of \eqref{eqn:main_err_term} (for $\eta=1$) and \eqref{eqn:KL_error} shows that the estimate \eqref{bd:KL26} is sharper than Theorem \ref{thm:ball-plunge} in both its dependence on $R$ and on $\varepsilon$. The two estimates also apply in different geometric settings: \eqref{bd:KL26} treats general pairs of domains with codimension-one boundaries, whereas Theorem \ref{thm:ball-plunge} assumes that one domain is a Euclidean ball but allows the other to have boundary of fractional codimension. We expect that the Schatten norm techniques of \cite{Kulikov-Larsen} may also extend to fractional-dimensional boundaries, although the precise form of such estimates remains unclear. The two approaches therefore appear to be complementary, relying on different techniques and currently applying in different geometric settings.

\subsection{Organization of the paper} 
The paper is organized as follows. In Section \ref{sec:notations}, we fix notation and recall the basic facts about Gevrey functions and frames. In Section \ref{wave-packet-on-ball},  we construct the wave packet system and prove that it is a frame for \(L^2(B_d(R))\).  In Section \ref{sect:energy-estimate}, we prove the Fourier localization estimates and the energy estimate for the system. In Section \ref{sec:proof-of-theorems}, we derive Theorems  \ref{thm:ball-frame-energy} and \ref{thm:ball-plunge} from the wave packet construction and the frame-counting lemma. Finally, in Section  \ref{sec:proof-of-corollary} we prove Corollary \ref{cor:Landau-quant-esti}.

\section{Definitions and notation}
\label{sec:notations}

Throughout the paper, $A \lesssim_{\alpha, \beta, \cdots } B$ means $A \leq CB$ for a constant $C$ depending only on parameters $\alpha, \beta, \cdots$. We write
$A\lesssim B$ to mean $A \lesssim_{d,s} B$.

For $f \in L^1(\mathbb{R}^d)\cap L^2(\mathbb{R}^d)$ with $d \geq 1$, we use the Fourier transform
\begin{equation}\label{def:FT}
\widehat{f}(\xi) = \frac{1}{(2\pi)^{d/2}} \int_{\mathbb{R}^d} f(x)e^{- i x\cdot \xi}dx,
\qquad
f(x) = 
\frac{1}{(2\pi)^{d/2}}\int_{\mathbb{R}^d} \widehat{f}(\xi)e^{i x\cdot \xi} \,d\xi,
\end{equation}
extended to $L^2$ by density. With this convention, Plancherel's theorem holds: $\|f\|_{2}=\|\widehat f\|_{2}$.

We denote by
\[
B_d(x,R):=\{y\in\R^d:|y-x|\leq R\}
\]
the closed Euclidean ball in $\R^d$ of radius $R$ centered at $x$. 

For \(a>0\), we write
\[
B_d(a):=\{y\in\R^d: |y|\leq a\},
\qquad
B_d:=B_d(1).
\]
Throughout the paper, the subscript indicates the ambient dimension, and we use this convention consistently to avoid ambiguity. 

Let $\mathbb{S}^{d-1}=\{\omega\in\R^{d}:|\omega|=1\}$.
The \emph{spherical cap} (\emph{geodesic ball} in $\mathbb{S}^{d-1}$),  with center $\omega_0\in\mathbb{S}^{d-1}$ and angular radius of $\delta > 0 $, is
\[
\kappa(\omega_0,\delta)
:=\{\omega\in\mathbb{S}^{d-1}:\ d_{\mathbb{S}^{d-1}}(\omega,\omega_0)\leq \delta\},
\]
where the geodesic distance is
\[
d_{\mathbb{S}^{d-1}}(\omega,\omega_0)=\arccos(\omega\cdot\omega_0).
\]

We write $\omega_\kappa$ to denote the center of a cap $\kappa$, and let $\delta_\kappa \in (0,\pi]$ denote the angular radius of $\kappa$. If $\kappa = \mathbb{S}^{d-1}$, then $\delta_\kappa = \pi$ and we make an arbitrary choice of $\omega_\kappa \in \mathbb{S}^{d-1}$.

A hemisphere in $\mathbb{S}^{d-1}$ is a cap of the form $\kappa(z,\pi/2)$. The antipodal hemispheres $\kappa(z,\pi/2), \kappa(-z,\pi/2)$ cover the sphere $\mathbb{S}^{d-1}$.

Given a hemisphere $\kappa(z,\pi/2) \subset \mathbb{S}^{d-1}$ there exists a real-analytic diffeomorphism $\Phi_z : \kappa(z,\pi/2) \rightarrow B_{d-1}(\pi/2)$, given by geodesic (normal) coordinates -- that is, $\Phi_z^{-1}$ is the exponential map restricted to a ball $B_{d-1}(\pi/2)$ in the tangent space $T_{z} \mathbb{S}^{d-1} \simeq \R^{d-1}$. Any function $f:\kappa(z,\pi/2)\to\C$ can be represented in geodesic coordinates by a function $\tilde f:B_{d-1}(\pi/2) \to\C$ given by
\[
\tilde f = f\circ \Phi_z^{-1}.
\]

The double of a cap $\kappa = \kappa(\omega_\kappa, \delta_\kappa)$, denoted by $\kappa^*$, is the cap with the same center and an angular radius of $2\delta_\kappa$. Formally:
$$\kappa^* := \kappa(\omega_\kappa,  2\delta_\kappa) = \{\omega \in \mathbb{S}^{d-1} : d_{\mathbb{S}^{d-1}}(\omega, \omega_\kappa) \leq 2\delta_\kappa\}.$$
Thus, if $\delta_\kappa\geq\pi/2$, then
$\kappa^*=\mathbb{S}^{d-1}$.

We equip $\mathbb{S}^{d-1}$ with its surface measure $d\sigma$. The surface measure of a cap $\kappa \subset \mathbb{S}^{d-1}$ of angular radius $\delta \in (0,\pi]$ satisfies $\sigma(\kappa)\approx \delta^{d-1}$. Therefore, $\approx \delta^{1-d}$ caps of angular radius $\delta$ suffice to cover $\mathbb{S}^{d-1}$.

\begin{lemma}
\label{lem:cap-chart-jacobian}
Given $z \in \mathbb{S}^{d-1}$, let $J_{z}(y)dy$ be the pushforward of the measure $d\sigma(\omega)|_{\kappa(z,\pi/2)}$ under the diffeomorphism $\Phi_{z} : \kappa(z,\pi/2) \rightarrow B_{d-1}(\pi/2)$. The diffeomorphism $\Phi_{z}$ is real analytic and moreover, there exist dimensional constants $C,c > 0$ so that for each multiindex $\alpha$,
\[
\sup_{y\in B_{d-1}(\pi/2)}\bigl|\partial_y^\alpha \Phi_{z}^{-1}(y)\bigr|\leq C \cdot C^{|\alpha|} \cdot \alpha !,
\]
and moreover 
\[
c \leq J_{z}(y)\leq C, \qquad y\in B_{d-1}(\pi/2),
\]
and
\[
\sup_{y\in B_{d-1}(\pi/2)}\bigl|\partial_y^\alpha J_{z}(y)\bigr|\leq C \cdot C^{|\alpha|} \cdot \alpha ! .
\]
\end{lemma}

\begin{proof}
Choose $\Phi_{z}^{-1}(y)=\exp_{z}(y)$ after identifying $T_{z} \mathbb{S}^{d-1}\simeq\R^{d-1}$.
Because $\pi/2$ is smaller than the injectivity radius of $\mathbb{S}^{d-1}$ and $\mathbb{S}^{d-1}$ is a real analytic manifold, the exponential map is real analytic, and its Jacobian is real analytic and bounded above and below on the fixed neighborhood $B_{d-1}(\pi/2)$. This implies the stated bounds. 
\end{proof}

\begin{lemma}
\label{lem:lattice_near_bd_min}
Let
\(E \subset \R^d\) be a closed set. Then
\begin{equation}\label{eqn:vol_bd}
\#\{{\bf m}\in\mathbb Z^d: \dist({\bf m},E) \leq t\}
\le
|E_{t+\sqrt d/2}|.
\end{equation}
\end{lemma}
\begin{proof}
Consider the unit cubes $Q_{\bf m} :={\bf m}+(-1/2,1/2]^d$ centered at ${\bf m}\in\mathbb Z^d$. These cubes are pairwise disjoint, and each has volume \(1\). If \(\dist({\bf m},E)\leq t\), then $Q_{\bf m}\subset E_{t+\sqrt d/2}$. Taking volumes gives the result.
\end{proof}

We present some basic properties of the variants of the upper Minkowski content introduced in Definition \ref{def:M_content}. It is easy to see that if $\cM^{*,d-\eta}(E)$ is finite, then $\diam(E) < \infty$. If $E$ contains at least two points and $\cM^{*,d-\eta}(E)$ is finite then $\cM^{d-\eta}(E)$ is finite. We also have the inequality
\[
 \cM^{*,d-\eta}(E) \leq \cM^{d-\eta}(E),
\]
as well as the scaling property: for \(t>0\),
\begin{align*}
\cM^{*,d-\eta}(t E)
= t^{d-\eta} \cM^{*,d-\eta}(E)
\qquad \text{and} \qquad
\cM^{d-\eta}(t E)
= t^{d-\eta} \cM^{d-\eta}(E).
\end{align*}

\subsection{Gevrey functions}
\label{sec:gevrey_intro} 
A critical aspect of our analysis is the use of Gevrey functions. We now recall their definition. Given a multiindex $\alpha \in \N_0^d$, we let $|\alpha| = \sum_j \alpha_j$ denote the order of $\alpha$, and $\alpha! = \prod_j \alpha_j !$.
\begin{definition}[Gevrey functions on $\R^d$ (\cite{Hor2,Rod1})]
For \(s \geq 1\) and an open set \(K \subset \R^d\), a $C^\infty$ function $\phi : K \rightarrow \C$ is said to be \emph{Gevrey--$s$} on \(K\) if there exist constants $A,B > 0$ such that for all multiindices $\alpha \in \N_0^d$, 
\begin{equation}
\| \partial^\alpha \phi \|_{L^\infty(K)} \leq A B^{|\alpha|} (\alpha !)^s.
\end{equation}
This definition extends to the case when $K$ is a closed Euclidean ball, as the partial derivatives of $\phi$ extend to the boundary of $K$ by continuity.
We refer to $(A,B)$ as the Gevrey constants of $\phi$. Write $\mathcal{G}^s(K:A,B)$ to denote the set of all scalar-valued Gevrey--$s$ functions $\phi: K \rightarrow \C$ with Gevrey constants $(A,B)$. We write $\cG^s(K,\R^n:A,B)$ for the set of all vector-valued functions $\Phi=(\Phi_i)_{1 \leq i \leq n}: K \rightarrow \R^n$ such that the coordinate function $\Phi_i$ is Gevrey--$s$ with constants $(A,B)$ for each $i$.
\end{definition}

A basic property of Gevrey functions is that their Fourier transforms decay at a near-exponential rate. Specifically, we have the following lemma (see Lemma A.2 of \cite{HIM-disk}). 

\begin{lemma}[Fourier decay of Gevrey functions]\label{lemma:Gevrey:Fourier_decay}
If $\phi \in \mathcal{G}^s(\R^d : A, B)$ and if $\supp(\phi)$ has Lebesgue measure at most $H$, then 
\begin{equation}\label{bd:Gevrey:FT}
|\widehat{\phi}(\xi)| \leq e^{d/2} A H \exp(- c (| \xi|/B)^{1/s}) 
\qquad  \text{for} \quad \xi \in \R^d
\end{equation}
with $c=\frac{1}{2e}$.
\end{lemma}

The next two lemmas assert that the class of Gevrey functions is closed under multiplication and composition.

\begin{lemma}\label{lemma:Gevrey:multiplication}
Let \(s_1,s_2 \geq 1\). 
For \(i=1,2\), let \(\phi_i \in \cG^{s_i}(\R^d:A_i,B_i)\). Then \(\phi_1 \cdot \phi_2 \in \cG^s(\R^d : A_3, B_3) \) with \(s = \max\{s_1,s_2\}\), \(A_3 = A_1 A_2\) and \(B_3=2 \max\{B_1,B_2\}\).
\end{lemma}

\begin{proof}
Set \(s=\max\{s_1,s_2\}\) and \(B=\max\{B_1,B_2\}\). For every multi-index
\(\alpha\in\mathbb N_0^d\), the Leibniz rule gives
\[
\partial^\alpha(\phi_1\phi_2)
=
\sum_{\beta\leq \alpha}
\binom{\alpha}{\beta}
(\partial^\beta \phi_1)(\partial^{\alpha-\beta}\phi_2).
\]
Hence,
\[
|\partial^\alpha(\phi_1\phi_2)(x)|
\le
A_1A_2
\sum_{\beta\leq \alpha}
\binom{\alpha}{\beta}
B_1^{|\beta|}B_2^{|\alpha-\beta|}
(\beta!)^{s_1}((\alpha-\beta)!)^{s_2}.
\]
Since \(s\ge s_1,s_2\) and \(B=\max\{B_1,B_2\}\),
\[
|\partial^\alpha(\phi_1\phi_2)(x)|
\le
A_1A_2 B^{|\alpha|}
\sum_{\beta\leq \alpha}
\binom{\alpha}{\beta}
(\beta!)^s((\alpha-\beta)!)^s.
\]
Using
\[
\binom{\alpha}{\beta}
=\frac{\alpha!}{\beta!(\alpha-\beta)!},
\]
we get
\[
\binom{\alpha}{\beta}
(\beta!)^s((\alpha-\beta)!)^s
=
(\alpha!)^s \binom{\alpha}{\beta}^{1-s}
\le
(\alpha!)^s,
\]
because \(s\ge 1\). Therefore,
\[
|\partial^\alpha(\phi_1\phi_2)(x)|
\le
A_1A_2 B^{|\alpha|}(\alpha!)^s
\#\{\beta:\beta\le \alpha\}.
\]
Finally,
\[
\#\{\beta:\beta\le \alpha\} = \prod_{j=1}^d(\alpha_j+1) \le 2^{|\alpha|}
.\]
Thus,
\[
|\partial^\alpha(\phi_1\phi_2)(x)|
\le
A_1A_2 (2B)^{|\alpha|}(\alpha!)^s.
\]
Hence, \(\phi_1\phi_2\) is Gevrey--\(s\), with constants \((A_1A_2,2\max\{B_1,B_2\})\). 
\end{proof}

\begin{lemma}\label{lemma:Gevrey:composition}
Let \(s \geq 1\) and \(A_0,A_1,B_0,B_1 > 0\) be fixed constants. 
If \(\psi\in \mathcal G^s(\R^N:A_0,A_1)\) and \(f\in \mathcal G^s(K,\R^N:B_0,B_1)\), then 
\(\psi\circ f\in \mathcal G^s(K:A_0,d^s B_1(1+N B_0 A_1))\).
\end{lemma}
\begin{proof}
Write \(f=(f_1,\dots,f_N)\). The case \(\alpha=0\) is immediate, since
\[
|(\psi\circ f)(x)|\leq A_0.
\]

Now let \(\alpha\neq 0\), and set \(m=|\alpha|\). By the multivariable
Faà di Bruno formula,
\[
\partial^\alpha(\psi\circ f)(x)
=
\sum_{k=1}^m
\sum_{\substack{\nu\in\mathbb N_0^N\\ |\nu|=k}}
(\partial^\nu\psi)(f(x))
\sum_{\lambda\in\Lambda(\alpha,\nu)}
\frac{\alpha!}
{
 \prod_{j=1}^N\prod_{\beta\neq 0}
 \lambda_{j,\beta}!(\beta!)^{\lambda_{j,\beta}}}
\prod_{j=1}^N\prod_{\beta\neq 0}
\bigl(\partial^\beta f_j(x)\bigr)^{\lambda_{j,\beta}},
\]
where \(\Lambda(\alpha,\nu)\) denotes the set of finitely supported
families
\[
\lambda
=
(\lambda_{j,\beta})_{
1\leq j\leq N,\,
\beta\in\mathbb N_0^d\setminus\{0\}}
\]
satisfying
\[
\sum_{\beta\neq 0}\lambda_{j,\beta}
=
\nu_j,
\qquad
j=1,\dots,N,
\]
and
\[
\sum_{j=1}^N\sum_{\beta\neq 0}
\lambda_{j,\beta}\beta
=
\alpha.
\]

For every \(\lambda\in\Lambda(\alpha,\nu)\), we have
\[
\sum_{j,\beta}\lambda_{j,\beta}
=
|\nu|
=
k
\]
and
\[
\sum_{j,\beta}\lambda_{j,\beta}|\beta|
=
|\alpha|
=
m.
\]
Using the Gevrey estimates for \(\psi\) and the coordinate functions
\(f_j\), we obtain
\[
\begin{aligned}
|\partial^\alpha(\psi\circ f)(x)|
&\leq
A_0B_1^m
\sum_{k=1}^m
(B_0A_1)^k
\sum_{\substack{\nu\in\mathbb N_0^N\\ |\nu|=k}}
(\nu!)^s
\sum_{\lambda\in\Lambda(\alpha,\nu)}
\frac{\alpha!}{\prod_{j,\beta}\lambda_{j,\beta}!}
\prod_{j,\beta}
(\beta!)^{(s-1)\lambda_{j,\beta}}.
\end{aligned}
\]

Fix \(k\), \(\nu\), and \(\lambda\in\Lambda(\alpha,\nu)\). Choose an
ordered list
\[
\beta_1,\dots,\beta_k
\]
in which each nonzero multiindex \(\beta\) occurs with multiplicity
\[
\sum_{j=1}^N\lambda_{j,\beta}.
\]
Then
\[
\beta_1+\cdots+\beta_k=\alpha.
\]
Set \(q_\ell=|\beta_\ell|\). Since \(q_\ell\geq 1\) and
\[
q_1+\cdots+q_k=m,
\]
by basic properties of multinomial coefficients we have
\[
\frac{m!}{q_1!\cdots q_k!}\geq k! 
\implies k!\prod_{\ell=1}^k q_\ell!\leq m! 
\implies k!\prod_{j,\beta} (\beta!)^{\lambda_{j,\beta}} \leq m!,
\]
where the last implication uses that \(\beta!\leq |\beta|!\) for every multiindex \(\beta\). 

Also,
\[
\frac{m!}{\alpha!}
=
\binom{m}{\alpha_1,\dots,\alpha_d}
\leq
d^m,
\]
since the multinomial coefficient on the left is one term in the
expansion of
\[
(1+\cdots+1)^m=d^m.
\]
It follows that
\[
(k!)^{s-1}
\prod_{j,\beta}
(\beta!)^{(s-1)\lambda_{j,\beta}}
\leq (m!)^{s-1} \leq 
d^{m(s-1)}(\alpha!)^{s-1}.
\]

Since \(|\nu|=k\), we also have \(\nu!\leq k!\). Because \(s\geq 1\), it follows that
\( (\nu!)^s = \nu!(\nu!)^{s-1} \leq k!(k!)^{s-1}\). Therefore,
\[
(\nu!)^s
\frac{\alpha!}{\prod_{j,\beta}\lambda_{j,\beta}!}
\prod_{j,\beta}
(\beta!)^{(s-1)\lambda_{j,\beta}}
\leq
d^{m(s-1)}(\alpha!)^s
\frac{k!}{\prod_{j,\beta}\lambda_{j,\beta}!}.
\]
Thus,
\[
\begin{aligned}
|\partial^\alpha(\psi\circ f)(x)|
&\leq
A_0d^{m(s-1)}B_1^m(\alpha!)^s
\sum_{k=1}^m
(B_0A_1)^k
\sum_{\substack{\nu\in\mathbb N_0^N\\ |\nu|=k}}
\sum_{\lambda\in\Lambda(\alpha,\nu)}
\frac{k!}{\prod_{j,\beta}\lambda_{j,\beta}!}.
\end{aligned}
\]

For fixed \(k\), the innermost double sum counts ordered lists
\[
\bigl((j_1,\beta_1),\dots,(j_k,\beta_k)\bigr)
\]
such that
\[
j_\ell\in\{1,\dots,N\},
\qquad
\beta_\ell\in\mathbb N_0^d\setminus\{0\},
\]
and
\[
\beta_1+\cdots+\beta_k=\alpha.
\]
Indeed, each family \(\lambda\) determines a multiset in which the pair $(j,\beta)$ occurs with multiplicity $\lambda_{j,\beta}$. The number of orderings of this multiset is the multinomial coefficient
\[
\frac{k!}{\prod_{j,\beta}\lambda_{j,\beta}!}.
\]
Conversely, every ordered list of pairs satisfying the stated conditions determines a unique family \(\lambda\) and hence a unique multiindex \(\nu\). Thus, the double sum counts precisely these ordered lists.

There are \(N^k\) choices for the labels \(j_1,\dots,j_k\). We next
bound the number of ordered decompositions
\[
\alpha=\beta_1+\cdots+\beta_k,
\qquad
\beta_\ell\neq 0.
\]
Encode each \(\beta_\ell\) by the word
\[
1^{\beta_{\ell,1}}
2^{\beta_{\ell,2}}
\cdots
d^{\beta_{\ell,d}}.
\]
Concatenating these \(k\) words produces a word of length \(m\) in an
alphabet of \(d\) symbols, together with \(k-1\) cuts between
consecutive symbols. This encoding is injective. Hence, the number of
such ordered decompositions is at most
\[
d^m\binom{m-1}{k-1}.
\]
Therefore,
\[
\sum_{\substack{\nu\in\mathbb N_0^N\\ |\nu|=k}}
\sum_{\lambda\in\Lambda(\alpha,\nu)}
\frac{k!}{\prod_{j,\beta}\lambda_{j,\beta}!}
\leq
N^kd^m\binom{m-1}{k-1}.
\]

It follows that
\[
|\partial^\alpha(\psi\circ f)(x)|
\leq
A_0d^{ms}B_1^m(\alpha!)^s
\sum_{k=1}^m
(NB_0A_1)^k
\binom{m-1}{k-1}.
\]
Finally,
\[
\sum_{k=1}^m
(NB_0A_1)^k 
\binom{m-1}{k-1}
=
NB_0A_1(1+NB_0A_1)^{m-1}
\leq
(1+NB_0A_1)^m.
\]
Consequently,
\[
|\partial^\alpha(\psi\circ f)(x)|
\leq
A_0
\bigl[d^sB_1(1+NB_0A_1)\bigr]^m
(\alpha!)^s.
\]
Thus,
\[
\psi\circ f
\in
\mathcal G^s
\bigl(
K:A_0,d^sB_1(1+NB_0A_1)
\bigr).
\]
\end{proof}

We can extend the definition of Gevrey functions to a real analytic manifold via local coordinates. We do so here for the sphere. For each \(x\in \mathbb{S}^{d-1}\), let \(\kappa(x,\pi/2) \subset \mathbb{S}^{d-1}\) denote the closed hemisphere centered at \(x\), and fix the diffeomorphism
\[
\Phi_x:\kappa(x,\pi/2)\to B_{d-1}(\pi/2)
\]
satisfying \(\Phi_x(x)=0\), arising from geodesic normal coordinates.

\begin{definition}[Gevrey functions on the sphere]
A function \(f:\mathbb{S}^{d-1}\to \C\) is called \emph{Gevrey--$s$} if there exist constants \(A,B>0\) such that, 
\[
\sup_{x\in \mathbb{S}^{d-1}}
\bigl\|\partial^\alpha (f\circ \Phi_x^{-1})\bigr\|_{L^\infty(B_{d-1}(\pi/2))} \leq A B^{|\alpha|} (\alpha !)^{s},
\]
for every multiindex \(\alpha\in \N_0^{d-1}\). 
We refer to \((A,B)\) as Gevrey--\(s\) constants of \(f\). We let $\cG^s(\mathbb{S}^{d-1} , \R^n : A,B)$ denote the set of vector-valued mappings $\Phi = (\Phi_i)_{1 \leq i \leq n}$ such that, for every $i \in \{1,\cdots,n\}$, $\Phi_i : \mathbb{S}^{d-1} \rightarrow \R$ is Gevrey--$s$ with constants $(A,B)$. Similarly, $\cG^s(\mathbb{S}^{d-1}:A,B)$ denotes the set of complex-valued
Gevrey--$s$ functions on $\mathbb{S}^{d-1}$ with constants $(A,B)$.
\end{definition}

Equivalently, \(f:\mathbb{S}^{d-1}\to \C\) is Gevrey--$s$ if, for every \(x\in \mathbb{S}^{d-1}\), the coordinate representation
\[
f\circ \Phi_x^{-1}: B_{d-1}(\pi/2)\to \C
\]
is Gevrey--\(s\) with constants uniform in \(x\). 

\begin{definition}\label{def:Gevrey:scale}
A family  $\{f_\nu\}_{\nu\in I}$ in a Hilbert space $\cH$ is a \emph{frame} for $\cH$ provided that 
\[
A \| f \|^2 \leq \sum_\nu |\langle f, f_\nu \rangle |^2 \leq B \| f \|^2 \mbox{ for all } f \in \cH.
\] 
Here, $0 < A \leq B < \infty$ are called frame constants for $\{ f_\nu \}$. 
If $\| f_\nu \| = 1$ for all $\nu$ then $\{f_\nu\}$ is called a \emph{unit-norm frame}. 
If $A=B$ then $\{f_\nu\}$ is called a \emph{tight frame}. 
A tight frame is \emph{Parseval} if $A=B=1$. 
\end{definition}
In this language, an orthonormal basis for $\cH$ is a unit-norm Parseval frame for $\cH$. We use the following elementary result from frame theory, which can be found in classical texts on the subject; for example, see \cite{christensen2003introduction}. 
\begin{lemma}\label{lem:frame_theory}
Suppose $T : \cH_1 \rightarrow \cH_2$ is a bounded invertible linear operator.
If $\{f_\nu\}$ is a frame for $\cH_1$ with frame constants $(A,B)$, then $\{ T f_\nu \}$ is a frame for $\cH_2$ with frame constants $(\| T^{-1} \|^{-2} A,\| T \|^2 B)$.
\end{lemma}

\section{Wave packet frames on the ball}
\label{wave-packet-on-ball}
In this section, we begin the proof of Theorem \ref{thm:ball-frame-energy}. Our goal is to construct a wave packet family adapted to the geometry of the ball $B_d(R)$. The construction parallels the construction on the disk \cite{HIM-disk}, but the boundary geometry now requires localization on spherical caps and a flattening procedure via local coordinate charts on $\mathbb{S}^{d-1}$.

We first make a remark on the invariances of the problem at hand. We fix a bounded measurable set $S \subset \R^d$ and a dyadic number $R \geq 2$ such that $\diam(S) R \geq 2$, as in the statement of Theorem \ref{thm:ball-frame-energy}. Note that the conclusion of the theorem is invariant under the rescalings $S \mapsto  2^{-J} S$ and $R \mapsto  2^J R$ for integer $J$. Indeed, if $\{\psi_\nu(x)\}$ is a frame for $L^2(B_d(R))$ satisfying the energy estimates \eqref{ineq:energy-estimate} with respect to the set $S$, then $\{ 2^{-Jd/2}\psi_\nu(2^{-J}x)\}$ is a frame for $L^2(B_d(2^J R))$ satisfying the energy estimates \eqref{ineq:energy-estimate} with respect to the set $2^{-J} S$. Furthermore, the error term that controls the exceptional family is preserved by this scaling, $H_\eta(2^{-J} S, 2^J R, \varepsilon) = H_\eta(S,R,\varepsilon)$. This is easily verified by using the homogeneity of the Minkowski content.

Therefore, by rescaling the set $S$ and $R$, if necessary, we may assume without loss of generality that $1/2 \leq \diam(S) \leq 1$. 

Given a vector $h \in \R^d$, we are also free to translate the set via $S \mapsto S + h$. Indeed, if $\{\psi_\nu(x)\}$ is a frame for $L^2(B_d(R))$ satisfying the energy estimates \eqref{ineq:energy-estimate} with respect to the set $S$, then the modulated family $\{ e^{i h \cdot x} \psi_\nu(x) \}$ is a frame for $L^2(B_d(R))$ satisfying the energy estimates \eqref{ineq:energy-estimate} with respect to the translated set $S + h$.

Thus, without loss of generality, we may assume that $S$ contains the origin. We record these assumptions as:
\[
S \subset B_d(1), \;\; 1/2 \leq \diam(S) \leq 1.
\]
We begin by decomposing $B_d(R)$ into radial--angular sectors of Whitney type.
We then choose Gevrey cutoffs adapted to this decomposition and use them to define
two classes of packets: interior packets whose phases are linear in Cartesian
coordinates and boundary packets whose phases are linear in the flattened
radial--tangential coordinates near $\partial B_d(R)$. We conclude by proving that
the resulting family is a unit-norm frame for $L^2(B_d(R))$ with frame bounds
independent of $R$ and of the Gevrey index $s$.

\subsection{A Whitney-type  decomposition of the ball} 
Consider the ball
\[
B_d(R):=\{x\in \R^d: |x|\leq R\}\subset \R^d,
\]
and assume \(R=2^{j_{\max}}\) for some integer \(j_{\max}\ge1\).

For each integer \(j\leq j_{\max}\), define the angular scale
\begin{equation}\label{delta_j:defn}
\delta_j:=\max\{ \pi 2^j/R, \pi/(2R)\}.
\end{equation}
Thus, \(\delta_j=\pi 2^j/R\) for \(0\leq j\leq j_{\max}\), while \(\delta_j=\pi/(2R)\) for \(j<0\).

We decompose the sphere $\mathbb{S}^{d-1}$ at scale $\delta_j$ as follows. Let \(\cK_j\) be a collection of spherical caps \(\kappa\subset \mathbb{S}^{d-1}\) of angular radius \(\delta_j\), whose centers form a maximal strictly \(\delta_j\)-separated subset of \(\mathbb{S}^{d-1}\); thus, the geodesic distance between distinct centers is larger than \(\delta_j\). Since \(\delta_{j_{\max}}=\pi\), \(\cK_{j_{\max}}\) consists of the single cap \(\kappa = \mathbb{S}^{d-1}\). 
The family \(\cK_j\) satisfies the following standard properties:
\begin{itemize}
\item The caps in \(\cK_j\) cover \(\mathbb{S}^{d-1}\).
\item Each \(\kappa\in \cK_j\) has angular radius \(\delta_j\).
\item The overlap of the family of doubled caps \(\{\kappa^*:\kappa\in\cK_j\}\) is uniformly bounded: there exists a constant \(C_d\) depending only on  the dimension \(d\) such that 
\[
\sum_{\kappa\in\cK_j}\mathbf 1_{\kappa^*}(\omega) \leq C_d, \qquad \omega\in\mathbb{S}^{d-1}.
\]
\end{itemize}
Because $\delta_j = \delta_{-1}$ for $j < 0$, we may assume
\[
\cK_j = \cK_{-1} \mbox{ for } j < 0.
\]

For \(j\leq j_{\max}\) and \(\kappa\in\cK_j\), define the sector
\[
S_{j,\kappa} := \left\{ x\in B_d(R): R-2^j\leq |x|\leq R-2^{j-1}, \quad \frac{x}{|x|}\in\kappa \right\},
\]
and its enlargement
\[
S_{j,\kappa}^* := \left\{   x\in B_d(R):R-(1.1)2^j\leq |x|\leq R-(0.9)2^{j-1},   \quad\frac{x}{|x|}\in\kappa^*   \right\}.
\]
For each fixed \(j\), the family \(\{S_{j,\kappa}:\kappa\in\cK_j\}\) covers the spherical shell
\[
\{x\in\R^d: R-2^j\leq |x|\leq R-2^{j-1}\},
\]
and the full collection
\[
\{S_{j,\kappa}: j\leq j_{\max},\ \kappa\in\cK_j\}
\]
covers \(B_d(R)\) up to a set of Lebesgue measure zero.
Moreover, the enlarged sectors have uniformly bounded overlap:
\[
\sum_{j\leq j_{\max}} \sum_{\kappa\in\cK_j} \mathbf 1_{S_{j,\kappa}^*}(x) \leq  C_d, \qquad x\in B_d(R).
\]
The spherical measure of a cap \(\kappa\in\cK_j\) satisfies
\[
\sigma(\kappa)\approx \delta_j^{d-1}.
\]
Using polar coordinates, it follows that
\begin{equation}\label{eqn:sect_meas}
|S_{j,\kappa}^*| \approx R^{d-1} 2^j \delta_j^{d-1}
\approx
\begin{cases}
2^{jd}, & 0\leq j\leq j_{\max},\\
2^j, & j<0.
\end{cases}
\end{equation}
We classify the sectors into two types. We call \(S_{j,\kappa}\) an
\textbf{interior sector} if \(0\leq j\leq j_{\max}\); in this case,
\[
S_{j,\kappa}\subset B_d(R-\tfrac12).
\]
We call \(S_{j,\kappa}\) a \textbf{boundary sector} if \(j<0\); in this case,
\[
S_{j,\kappa}\subset B_d(R)\setminus B_d(R-\tfrac12),
\]
so that the sector is near the boundary of \(B_d(R)\). 
\subsection{Radial and angular cutoffs}
For each integer \(j\leq j_{\max}\), define the interval
\begin{equation}\label{radial_int}
I_j := \bigl[R-1.1\cdot 2^{j}, R-0.9\cdot 2^{j-1}\bigr]\cap[0,R].
\end{equation}
Fix \(s>1\). Choose a family of nonnegative radial cutoffs
\[
\{\phi_j\}_{j\leq j_{\max}}\subset C^\infty([0,R])
\]
satisfying:
\begin{itemize}
\item[(R1)]
\(\displaystyle \sum_{j\leq j_{\max}} \phi_j(r)^2 \equiv 1\) for \(r\in[0,R)\),
\item[(R2)]
\(
\displaystyle \supp(\phi_j)\subset I_j,
\)
\item[(R3)]
Each \(\phi_j\) belongs to 
$\cG^s(I_j : C_1,C_2 2^{-j})$; 
this means
\[
\|\partial_r^k \phi_j\|_{L^\infty} \leq C_1 C_2^k (k!)^s 2^{-jk}
\qquad \text{for all} \quad k\ge0, 
\]
where the constants \(C_1,C_2>0\) depend only on \(s\) and are independent of \(j\) and \(k\).
\item[(R4)] 
For each \(j\leq j_{\max}\), there exists a subinterval \(I_j^\circ\subset I_j\) of comparable length, say \(|I_j^\circ| \geq |I_j|/2\), such that
\[
\phi_j(r)\geq c>0 \quad\text{on } I_j^\circ
\]
where \(c\) is an absolute constant (independent of \(j\) and \(s\)).
\end{itemize}
For the construction of these cutoffs, see Lemma~B.1 of Appendix~B in \cite{HIM-disk}.

Property (R2) implies that only nearest-neighbor overlaps occur:
\[
\supp(\phi_j)\cap \supp(\phi_{j'}) \neq \varnothing \quad  \Longrightarrow \quad  |j-j'| \leq 1.
\]
Since \(2^{j_{\max}}=R\), we have
\[
I_{j_{\max}}=[0,0.55R], \qquad I_j\subset [0.45R,R], \quad \text{for } j<j_{\max}.
\]
Consequently, (R1) and (R2) imply that
\begin{equation}\label{eqn:phi_top_scale}
\phi_{j_{\max}}(r)=1, \qquad r\in [0,0.45R].
\end{equation}
Similarly, for each \(j<j_{\max}\),
\[
\phi_j(r)=1, \qquad r\in I_j^\circ,
\]
where we take 
\[
I_j^\circ := [R-0.9\cdot 2^j, R-1.1\cdot 2^{j-1}].
\]
For each \(j\leq j_{\max}\), choose a family of nonnegative angular cutoffs
\[
\{\eta_{j,\kappa}\}_{\kappa\in\cK_j} \subset C^\infty(\mathbb{S}^{d-1})
\]
satisfying:
\begin{itemize}
\item[(A1)]
\(\displaystyle \sum_{\kappa\in\cK_j} \eta_{j,\kappa}(\omega)^2 \equiv 1, \qquad \omega\in\mathbb{S}^{d-1},\)
\item[(A2)]
\( \displaystyle \supp(\eta_{j,\kappa})\subset \kappa^*, \qquad \kappa\in\cK_j,\)
\item[(A3)]
Each \(\eta_{j,\kappa}\) is in \(\cG^s(\mathbb{S}^{d-1} : C_1,C_2\delta_j^{-1})\). Equivalently, for every multiindex \(\alpha\in\mathbb Z_{\ge0}^{d-1}\) and every hemispherical coordinate chart $\Phi_z : \kappa(z,\pi/2) \rightarrow B_{d-1}(\pi/2)$,
\[
\sup_{x\in B_{d-1}(\pi/2)} \left| \partial_x^\alpha \bigl(\eta_{j,\kappa}\circ\Phi_z^{-1}\bigr)(x) \right| \leq C_1 C_2^{|\alpha|} (\alpha!)^s \delta_j^{-|\alpha|}.
\]
where the constants $C_1,C_2 > 0$ depend only on $s$ and $d$ (and are uniform in $z$).
\item[(A4)] For each \(\kappa\in \cK_j\), there exists a cap \(\kappa^\circ\subset \kappa\) of comparable radius such that
\[
\eta_{j,\kappa}\geq c>0 \quad\text{on }\kappa^\circ,
\]
where $c$ is a dimensional constant.
\end{itemize}
At the top scale \(j=j_{\max}\), the partition is trivial, hence:
\[
\eta_{j_{\max},\mathbb{S}^{d-1}}
\equiv 1.
\]
Because $\cK_j = \cK_{-1}$ for $j < 0$, we may assume that the partition of unity is independent of $j$ for $j < 0$. That is, $\eta_{j,\kappa} = \eta_{-1,\kappa}$ for $j < 0$.

The construction of such cutoff families is standard; see \cite{HIM-disk} for the case $d=2$. The same argument extends to higher dimensions by replacing arcs on \(\mathbb{S}^1\) with spherical caps on \(\mathbb{S}^{d-1}\) and using geodesic normal coordinates.

\subsection{Wave packets in $L^2(B_d(R))$} 
The wave packets are indexed by an integer \(j\leq j_{\max}\), specifying a radial shell; a cap \(\kappa\in\cK_j\); and a lattice index \(\mathbf{m}\in\Z^d\), specifying the modulation parameters. We denote the full index set by
\begin{equation}\label{eqn:indexset}
\cI
:=
\left\{
\nu=(j,\kappa,\mathbf{m})
:
j\leq j_{\max},\ 
\kappa\in\cK_j,\ 
\mathbf{m}\in\Z^d
\right\}.
\end{equation}

For \(0\leq j\leq j_{\max}\), we define the \textbf{interior wave packets} by
\begin{equation}\label{int_wp}
\psi_{j,\kappa,\mathbf{m}}(x)
=
C_{j,\kappa}\,2^{-jd/2}
\phi_j(r)\eta_{j,\kappa}(\omega)
e^{ic_\ast 2^{-j}\mathbf{m}\cdot x},
\qquad x=r\omega,
\end{equation}
where \(r\geq0\), \(\omega\in\mathbb{S}^{d-1}\), and \(c_\ast>0\) is a sufficiently small constant to be chosen later. The constant \(C_{j,\kappa}>0\) is chosen so that
\[
\|\psi_{j,\kappa,\mathbf{m}}\|_{L^2(\R^d)}=1.
\]
Since the modulation factor has modulus one, \(C_{j,\kappa}\) is independent of \(\mathbf{m}\). Here, \(\mathbf{m}\in\Z^d\) parametrizes the linear phase modulation in Cartesian coordinates.

We next define the boundary packets. Given a cap
\[
\kappa=\kappa(z,\delta),
\qquad
0<\delta\leq\frac{\pi}{2},
\]
let \(\Phi_\kappa\) denote the restriction to \(\kappa\) of the geodesic normal coordinate map
\[
\Phi_z:\kappa\left(z,\frac{\pi}{2}\right)
\longrightarrow
B_{d-1}\left(\frac{\pi}{2}\right).
\]
Thus,
\[
\Phi_\kappa:\kappa\longrightarrow B_{d-1}(\delta)
\]
is the geodesic normal coordinate map centered at \(z\).

For \(j<0\), write
\[
\mathbf{m}=(m_1,\mathbf{m}')
\in\Z\times\Z^{d-1}.
\]
We define the \textbf{boundary wave packets} by
\begin{equation}\label{bdry_wp}
\psi_{j,\kappa,\mathbf{m}}(r\omega)
=
\begin{cases}
\displaystyle
C_{j,\kappa}\,2^{-j/2}
\phi_j(r)\eta_{j,\kappa}(\omega)
\exp\left[
\pi i\left(
m_1 2^{-j}r
+
\mathbf{m}'\cdot
\frac{R}{\pi}\Phi_{\kappa^*}(\omega)
\right)
\right],
& \omega\in\kappa^*,\\[1.2em]
0,
& \omega\notin\kappa^*.
\end{cases}
\end{equation}
As above, \(C_{j,\kappa}>0\) is chosen so that
\[
\|\psi_{j,\kappa,\mathbf{m}}\|_{L^2(\R^d)}=1.
\]

For $j<0$, $\kappa \in \mathcal{K}_j$, the cap $\kappa^*$ has angular radius $2 \delta_j = \pi/R$ and $\Phi_{\kappa^*}$ is a geodesic coordinate chart on $\kappa^*$. Therefore,
\[
\frac{R}{\pi}\Phi_{\kappa^*}(\omega)
\in B_{d-1}(1)
\qquad
\text{for }\omega\in\kappa^*.
\]
In \eqref{bdry_wp}, the integer \(m_1\) parametrizes modulation in the radial variable, while \(\mathbf{m}'\in\Z^{d-1}\) parametrizes modulation in the tangential geodesic coordinates.

By the support properties of \(\phi_j\) and \(\eta_{j,\kappa}\),
\[
\operatorname{supp}(\psi_{j,\kappa,\mathbf{m}})
\subset S_{j,\kappa}^*
\subset B_d(R),
\qquad
j\leq j_{\max},
\quad
\kappa\in\cK_j.
\]

We finally record uniform bounds for the normalization constants. For
\(0\leq j\leq j_{\max}\),
\[
|S_{j,\kappa}^*|\approx_d 2^{jd},
\]
whereas for \(j<0\),
\[
|S_{j,\kappa}^*|\approx_d 2^j.
\]
Thus, the prefactors \(2^{-jd/2}\) and \(2^{-j/2}\) compensate for the volumes of the corresponding supports. Moreover, the cutoff
\[
\phi_j(r)\eta_{j,\kappa}(\omega)
\]
is bounded above by \(1\) and bounded below by a positive constant on a subset of \(S_{j,\kappa}^*\) whose measure is comparable to \(|S_{j,\kappa}^*|\). Since the phase factors have modulus one, it follows that
\[
\|\psi_{j,\kappa,\mathbf{m}}\|_{L^2(\R^d)}^2
\approx_d C_{j,\kappa}^2.
\]
The normalization
\[
\|\psi_{j,\kappa,\mathbf{m}}\|_{L^2(\R^d)}=1
\]
therefore implies
\[
0<c_d\leq C_{j,\kappa}\leq C_d<\infty.
\]

\subsection{Frame property of the wave packet system}
We now prove that the wave packets constructed above form a unit-norm frame for
$L^2(B_d(R))$. The proof has two parts. For the interior packets, each enlarged sector is contained in a cube of sidelength comparable to $2^j$, and the packet family is obtained by restricting a Fourier basis on that cube.
For the boundary packets, each boundary sector is flattened by means of a
coordinate chart on $B_{d-1}$. We then compare the induced weighted
$L^2$-space with an unweighted Euclidean $L^2$-space, and finally transport a
Fourier system back to the curved sector using the nonlinear coordinate chart.
\begin{lemma}\label{lem:frame}
If $c_\ast \in (0,1)$ in \eqref{int_wp} is chosen as a sufficiently small dimensional constant, then the system \(\{\psi_{j,\kappa,\mathbf{m}}\}\) forms a unit-norm frame for \(L^2(B_d(R))\), with frame bounds $0 < A \leq B < \infty$ that are dimensional constants, independent of $s$ and $R$.
\end{lemma}

\begin{proof} Consider the partition of unity at level $j\leq j_{\max}$ on the sphere $\mathbb{S}^{d-1}$,
\begin{align}\label{partition_of_unity_on_sphere}
\sum_{\kappa \in  \cK_j }\eta_{j,\kappa}(\omega)^2  \equiv 1 \qquad  ( \omega \in   \mathbb{S}^{d-1}).
\end{align}
Also consider the partition of unity on $[0,R)$,
\begin{align}\label{partition_of_unity_on_interval}
\sum_{j \leq j_{\max}}\phi_{j}(r)^2  \equiv 1 \qquad  ( r \in  [0,R)).
\end{align}
Here, we consider $r = |x|$ and $\omega = x/|x|$ as the radial coordinates of $x \in \R^d \setminus \{0\}$. Then for any $f\in L^2(B_d(R), dx)$, 
\begin{align}\label{eqn:norm0}
\|f\|_{L^2(B_d(R),dx)}^2
= \int_{B_d(R)} |f(x)|^2 \left(\sum_{j,\kappa} \phi_j(r)^2 \eta_{j,\kappa}(\omega)^2 \right) dx  
= \sum_{j,\kappa} \| f \phi_j \eta_{j,\kappa} \|_{L^2(S_{j,\kappa}^*,dx)}^2.  
\end{align}

{\it Frame property of the interior wave packets:} Fix a patch index $(j,\kappa)$ with $0\leq j\leq j_{\max}$ and $\kappa\in \cK_j$.
The cutoff function $\phi_j(r) \eta_{j,\kappa}(\omega)$ (considered as a function of $x$) is supported on the region $S_{j,\kappa}^*\subset B_d(R)$, where
\[
S_{j,\kappa}^*
:=
\left\{
x = r \omega \in B_d(R):
r \in I_j, \omega\in \kappa^*
\right\},
\]
where the interval $I_j = [R-1.1\cdot 2^j,\;R-0.9\cdot 2^{j-1}]$ contains the support of $\phi_j(r)$, and
the spherical cap $\kappa^*$ contains the support of $\eta_{j,\kappa}(\omega)$.

Fix
\[
r_j:=R-\tfrac342^j \in I_j,
\qquad
x_{j,\kappa}:=r_j\omega_\kappa \in S_{j,\kappa}^*.
\]
If $x=r \omega \in S_{j,\kappa}^*$, then $d_{\mathbb{S}^{d-1}}(\omega, \omega_\kappa) \leq 2 \delta_j$ (as $\omega \in \kappa^*$ and $\kappa^*$ has radius $2 \delta_j$) and $|r-r_j| \leq   2^j$. Therefore,
\[
|x-x_{j,\kappa}| = | r \omega - r_j \omega_\kappa|  \leq |r-r_j| |\omega| + r_j |\omega - \omega_\kappa| \lesssim 2^j + R \delta_j  \lesssim 2^j.
\]
We deduce that $S_{j,\kappa}^*$ is contained in a ball of radius $\approx 2^j$ centered at $x_{j,\kappa}$, which is contained in a cube of diameter $\approx 2^j$ centered at $x_{j,\kappa}$. Thus,  there exists a constant $C_\ast = C_\ast(d) >1$, such that
\[
S_{j,\kappa}^*\subset T_{j,\kappa}:=
x_{j,\kappa}+[-C_\ast2^{j-1},C_\ast2^{j-1}]^d.
\]
Consider the orthonormal Fourier basis for \(L^2(T_{j,\kappa})\): 
\[
\bigl\{ C_\ast^{-d/2} 2^{-jd/2} e^{2\pi i2^{-j} \mathbf{m} \cdot x/ C_\ast} : \mathbf{m}\in\Z^d \bigr\}. 
\] 
By Parseval's identity, 
\[
\begin{aligned}
\| f\phi_j\eta_{j,\kappa} \|_{L^2(S_{j,\kappa}^*,dx)}^2
&= \sum_{\mathbf{m}\in\Z^d}
\bigl|\bigl\langle f\phi_j\eta_{j,\kappa},
C_\ast^{-d/2} 2^{-jd/2} e^{2\pi i2^{-j} \mathbf{m} \cdot x / C_\ast}\bigr\rangle\bigr|^2 
\\
&= \sum_{\mathbf{m}\in\Z^d}
\bigl|\bigl\langle f,
C_\ast^{-d/2} \phi_j\eta_{j,\kappa} 2^{-jd/2} e^{2\pi i2^{-j} \mathbf{m} \cdot x / C_\ast}\bigr\rangle\bigr|^2 
\\
&= C_\ast^{-d} \sum_{\mathbf{m}\in\Z^d} |\langle f, C_{j,\kappa}^{-1} \psi_{j,\kappa,\mathbf{m}}\rangle|^2,
\end{aligned}
\]
where the \(\psi_{j,\kappa,\mathbf{m}}\) are defined in \eqref{int_wp} with $c_\ast = 2 \pi C_\ast^{-1}$. We may assume $C_\ast > 2\pi$ in the above, and hence, $c_\ast \in (0,1)$ as required. Because $0 < c \leq C_{j,\kappa} \leq C$ for dimensional constants $c,C$,  we obtain
\begin{equation}\label{eqn:norm1}
c_0 \| f \phi_j \eta_{j,\kappa} \|_{L^2(S_{j,\kappa}^*,dx)}^2 \leq \sum_{\mathbf{m} \in \Z^d} |\langle f, \psi_{j,\kappa,\mathbf{m}} \rangle|^2 \leq C_0 \| f \phi_j \eta_{j,\kappa} \|_{L^2(S_{j,\kappa}^*,dx)}^2
\end{equation}
for dimensional constants $c_0,C_0 > 0$.

{\it Frame property of the boundary wave packets:} Fix a patch index $(j,\kappa)$ with $j<0$ and $\kappa\in \cK_j$.  
Let $\kappa^*$ denote the enlarged cap of angular radius $2 \delta_j = \pi/R$, and let $\Phi_{\kappa^*}:\kappa^*\to B_{d-1}(\pi/R)$ be the standard  diffeomorphism arising from geodesic coordinates. We denote $J_{\kappa^*}(y)dy$ for the pushforward of the surface measure $d \sigma$ on $\kappa^*$ by the diffeomorphism $\Phi_{\kappa^*}$. 
Thus, for every integrable function $h$ on $\kappa^*$ we have 
\[ 
\int_{\kappa^*}  h(\omega) d \sigma(\omega)
=
\int_{B_{d-1}(\pi/R)} h(\Phi_{\kappa^*}^{-1}(y))  J_{\kappa^*}(y)dy.
\] 
Hence, for every integrable function $g$ on $S_{j,\kappa}^*$,
\[ 
\int_{S_{j,\kappa}^*}  g(x) dx 
=
\int_{I_j} \int_{\kappa^*}  g(r\omega) r^{d-1} d \sigma(\omega) dr 
=
\int_{I_j}\int_{B_{d-1}(\pi/R)} g(r \Phi_{\kappa^*}^{-1}(y)) r^{d-1}  J_{\kappa^*}(y)dydr.
\] 

In particular, applying this identity to $|g|^2$ shows that the coordinate transformation gives rise to an isometry
\[
U:L^2(S_{j,\kappa}^*,dx) \to L^2\bigl(I_j\times  B_{d-1}(\pi/R),r^{d-1}J_{\kappa^*}(y)\,dr\,dy\bigr),
\qquad
(Ug)(r,y)=g\bigl(r\Phi_{\kappa^*}^{-1}(y)\bigr).
\]
Because $J_{\kappa^*}(y) \approx_d 1$ and $r \approx_d R$ on $I_j$  ($j<0$), the weight $ r^{d-1} J_{\kappa^*}(y) \approx_d R^{d-1}$ on $I_j$. Therefore, 
\begin{equation}\label{id_isom}
\operatorname{id}: L^2(I_j \times B_{d-1}(\pi/R), r^{d-1}J_{\kappa^*}(y) dr dy) \to L^2(I_j \times B_{d-1}(\pi/R), R^{d-1} dr dy)
\end{equation}
is an isomorphism, with the norm of this isomorphism and the norm of its inverse bounded by dimensional constants.

Choose an enlarged box  
$H_j \supset I_j\times B_{d-1}(\pi/R)$ in $\R^d$ such that  
\[
H_j := I_j^+ \times Q,
\]
for an interval $I_j^+ \supset I_j$ of size $|I_j^+| = 2^{j+1}$ and an axis-parallel cube $Q \supset B_{d-1}(\pi/R)$ of sidelength $\delta_Q = 2 \pi/R$.

Let $\cB:=\{h_{\mathbf{m}}\}_{{\bf m}\in \mathbb{Z}^d}$ be the orthonormal Fourier basis for  $L^2(H_j, R^{d-1} drdy)$. Thus, 
\[
h_{\mathbf{m}}(r,y) = R^{(1-d)/2} |I_j^+|^{-1/2}|Q|^{-1/2} \exp\Bigg( 2\pi i \Big(\frac{m_1 r}{2^{j+1}}
+\frac{(\mathbf{m}'\cdot y) R }{ 2 \pi}\Big)\Bigg),
\]
where $\mathbf{m} = (m_1,\mathbf{m}') \in \Z \times \Z^{d-1}$.  
Let $\widetilde h_{\bf m}$ denote the restriction of $h_{{\bf m}}$ to $I_j\times   B_{d-1}(\pi/R)$.
Then $\{\widetilde h_{{\bf m}}\}_{{{\bf m}}\in\mathbb \Z^d}$ is a Parseval frame for $L^2(I_j\times  B_{d-1}(\pi/R), R^{d-1} dr\,dy)$ (viewed as a subspace of $L^2(H_j, R^{d-1} drdy)$). That is, for every $f$,
\[
\|f\|_{L^2(I_j\times B_{d-1}(\pi/R), R^{d-1} drdy)}^2 = \sum_{\mathbf{m} \in \Z^d} |\langle f, \widetilde{h}_{\mathbf{m}}\rangle|^2.
\]
Because \eqref{id_isom} is an isomorphism, and by Lemma \ref{lem:frame_theory}, the restricted Fourier system $\{\widetilde h_{{\bf m}}\}_{{{\bf m}}\in\mathbb \Z^d}$ is therefore a frame for the weighted space $L^2(I_j\times B_{d-1}(\pi/R), r^{d-1} J_{\kappa^*}(y) drdy)$, with frame bounds  $0 < c < C < \infty$ that are dimensional constants. That is, for any $f$,
\[
c \|f\|_{L^2(I_j\times B_{d-1}(\pi/R), r^{d-1} J_{\kappa^*}(y) drdy)}^2 \leq \sum_{\mathbf{m} \in \Z^d} |\langle f, \widetilde h_{\mathbf{m}}\rangle|^2 \leq C \|f\|_{L^2(I_j\times B_{d-1}(\pi/R), r^{d-1} J_{\kappa^*}(y) drdy)}^2.
\]
Pulling this frame back via the isometry $U^{-1}$, define 
\[
f_{{\bf m}}(x)
=
\widetilde h_{{\bf m}}\bigl(r,\Phi_{\kappa^*}(\omega)\bigr),
\qquad x=r\omega\in S_{j,\kappa}^*.
\]
Then the frame condition asserts that, for any   $f\in L^2(S_{j,\kappa}^*,dx)$, 
\[
c \|f\|_{L^2(S_{j,\kappa}^*, dx)}^2 \leq \sum_{\mathbf{m} \in \Z^d} |\langle f, f_{\mathbf{m}}\rangle|^2 \leq C \|f\|_{L^2(S_{j,\kappa}^*, dx)}^2.
\]
Replacing $f$ by $f \phi_j \eta_{j,\kappa}$, we obtain 
\[
c \|f \phi_j \eta_{j,\kappa}\|_{L^2(S_{j,\kappa}^*, dx)}^2 \leq  \sum_{\mathbf{m} \in \Z^d} |\langle f, \phi_j \eta_{j,\kappa} f_{\mathbf{m}}\rangle|^2 \leq C \|f \phi_j \eta_{j,\kappa}\|_{L^2(S_{j,\kappa}^*, dx)}^2.
\]
Moreover, 
\[ 
(\phi_j \eta_{j,\kappa} f_{\mathbf{m}})(x)= 
\overline{a} \phi_j(r) \eta_{j,\kappa}(\omega) 2^{-j/2} \exp\left( \pi i \left(m_1r 2^{-j} + \mathbf{m}' \cdot \Phi_{\kappa^*}(\omega) R/\pi\right)\right), \;\;\; x = r \omega \in S_{j,\kappa}^*,
\]
for an absolute constant $\overline{a} > 0$. Therefore, $$\phi_j \eta_{j,\kappa} f_{\mathbf{m}} = \overline{a} C_{j,\kappa}^{-1} \psi_{j,\kappa,\mathbf{m}}.$$ 
It follows that 
\begin{equation}\label{eqn:norm2}
c_1 \|f \phi_j \eta_{j,\kappa}\|_{L^2(S_{j,\kappa}^*, dx)}^2 \leq  \sum_{\mathbf{m} \in \Z^d} |\langle f, \psi_{j,\kappa,\mathbf{m}}\rangle|^2 \leq C_1 \|f \phi_j \eta_{j,\kappa}\|_{L^2(S_{j,\kappa}^*, dx)}^2.
\end{equation}
Using \eqref{eqn:norm0}, and summing \eqref{eqn:norm1} and \eqref{eqn:norm2} over the relevant $j,\kappa$, we deduce that
\[c \| f \|^2_{L^2(B_d(R))} \leq \sum_{j,\kappa,\mathbf{m}} | \langle f, \psi_{j,\kappa,\mathbf{m}} \rangle |^2 \leq C \| f \|^2_{L^2(B_d(R))}.\]
\end{proof}
\section{Fourier localization and energy concentration estimates}\label{sect:energy-estimate}
The purpose of this section is to prove the following energy estimates. 
Throughout this section, \(S\subset \R^d\) is a bounded measurable set satisfying the constraints in the proposition below.
\begin{proposition}[Energy concentration of wave packets on the ball]\label{prop:energy-estimate}
Let $d \geq 2$, let $s > 1$, and let $R \geq 2$ be dyadic. Let $\{\psi_\nu\}_{\nu\in \cI}$ be the Gevrey--$s$ wave packet system on $B_d(R)$ constructed in Section ~\ref{wave-packet-on-ball}.
Let \(S\subset B_d(1)\) satisfy $1/2 \leq \diam(S) \leq 1$, and $\cM^{d-\eta}(\partial S)<\infty$ for some \(0<\eta\leq 1\). Then, for each $\varepsilon \in (0,1/2]$, there exists a decomposition of the index set
\[
\cI = \cI_1 \cup \cI_2 \cup \cI_3
\]
such that
\[
\sum_{\nu \in \cI_1} \|\widehat{\psi_\nu}\|_{L^2(S)}^2 + \sum_{\nu \in \cI_2} \|\widehat{\psi_\nu}\|_{L^2(\mathbb{R}^d \setminus S)}^2
\leq \varepsilon^2,
\]
where
\[
\#(\cI_3) \lesssim_{d,s,\eta} \left\{
\begin{aligned}
&\cM^{d-\eta}(\partial{S})
\biggl( R^{d-1} \log(R/\varepsilon)^{sd+1} + R^{d-\eta} \log(R/\varepsilon)^{s\eta} \biggr), \quad 0 < \eta < 1 \\
&\cM^{d-1}(\partial S) R^{d-1} \log(R/\varepsilon)^{sd+1}, \quad \eta=1
\end{aligned}
\right.
\]
\end{proposition}
We will prove this result in \S \ref{proof-of-prop:energy-estimate}. 
\begin{remark}
The decomposition $\mathcal I=\mathcal I_1\cup\mathcal I_2\cup\mathcal I_3$ depends on both $S$ and the choice of $\varepsilon$.
To keep the notation simple, we suppress this dependence in the notation.
\end{remark}
\subsection{Analysis of interior wave packets}
For $0\leq j\leq j_{\max}$ and $\kappa\in \cK_j$, set
\begin{equation}\label{eq:def-hjk}
h_{j,\kappa}(x)=\phi_j(|x|)\eta_{j,\kappa}(x/|x|),
\qquad x\in\R^d\setminus\{0\}.
\end{equation}
At the top scale $j=j_{\max}$, the angular partition is trivial and
$\kappa=\mathbb{S}^{d-1}$. Lemma~\ref{derivative--bounds--ball} below shows that
$h_{j,\kappa}$ has a smooth extension to $x=0$ and satisfies uniform Gevrey
estimates on $\R^d$.
\begin{lemma}\label{derivative--bounds--ball}
Let $s>1$, $0\leq j\leq j_{\max}$, and  $\kappa\in \cK_j$. Then $h_{j,\kappa}$ extends to  a  $C^\infty$ function on $\R^d$. Moreover, for every multi-index $\alpha \in \mathbb{Z}_{\geq 0}^d$,
\[
\bigl|\partial_x^\alpha h_{j,\kappa}(x)\bigr|
\leq A B^{|\alpha|} 2^{-j|\alpha|} (\alpha!)^s,
\qquad x \in \mathbb{R}^d,
\]
where $A,B > 0$ depend only on $d$ and $s$. That is, $h_{j,\kappa} \in \cG^s(\R^d : A,B 2^{-j})$. 
\end{lemma}
\begin{proof}
We first assume that $0\leq j<j_{\max}$. Then
$ \supp(\phi_j)\subset I_j\subset [0.45R,R]\subset [R/4,R]$, so the function $h_{j,\kappa}$ is supported in the annulus
$$
\cA_R=\{x\in\R^d:R/4\leq |x|\leq R\}.
$$
Recall a function $\varphi$ on $\R^d$ is said to be $k$-homogeneous if $\varphi(tx) = t^k \varphi(x)$. If $\varphi$ is in $\cG^s(\cA_1:C_0,C_1)$ then $\varphi_R(\cdot) := \varphi(\cdot/R)$ is in $\cG^s(\cA_R: C_0,C_1 R^{-1})$. If $\varphi$ is also $k$-homogeneous, then $\varphi = R^k \varphi_R$ is in $\cG^s(\cA_R: C_0 R^k, C_1 R^{-1})$.

On the annulus $\cA_1$, the map $x\mapsto |x|$ is real analytic, therefore it is in $\cG^1(\cA_1:C,C)$ for some $C$. This map is also $1$-homogeneous. Thus, $x \mapsto |x|$ is in $\cG^1(\cA_R: CR,CR^{-1})$. 

According to condition (R3), $\phi_j \in \cG^s(\R:A_0,A_1 2^{-j})$. By the composition law of Gevrey classes (see Lemma \ref{lemma:Gevrey:composition}), $x \mapsto \phi_j(|x|)$ is in $\cG^s(\cA_R: C, CR^{-1}(1+R2^{-j}))$. Because $R \geq 2$ and $2^j \leq R$, we deduce that $x \mapsto \phi_j(|x|)$ is in $\cG^s(\cA_R: C, C 2^{-j})$. 

Choose two points $z_1,z_2\in \mathbb{S}^{d-1}$ such that the hemispheres $\kappa(z_\ell,\pi/2)$, $\ell=1,2$, cover $\mathbb{S}^{d-1}$. Define
$$
\cA_{R,\ell}
=
\{x\in \cA_R: x/|x|\in \kappa(z_\ell,\pi/2)\}, \;\; \ell = 1,2.
$$
For each $z_\ell$, fix a hemispherical coordinate chart $\Phi_{z_\ell} : \kappa(z_\ell,\pi/2) \rightarrow B_{d-1}(\pi/2)$. 
The map $x\mapsto \Phi_{z_\ell}(x/|x|)$ is real analytic on $\cA_{R,\ell}$. It is also homogeneous of degree zero. Therefore, by the rescaling argument above, the map is in $\cG^1(\cA_{R,\ell}: C,CR^{-1})$.

According to condition (A3), $\eta_{j,\kappa} \in \cG^s(\mathbb{S}^{d-1} :B_0,B_1\delta_j^{-1})$. By definition, Gevrey classes on a manifold are defined through local coordinates. Thus, $\eta_{j,\kappa}\circ\Phi_{z_\ell}^{-1} \in \cG^s(B_{d-1}(\pi/2) : B_0,B_1\delta_j^{-1})$ for each $\ell$. By the composition law of Gevrey classes (see Lemma \ref{lemma:Gevrey:composition}), the function $x\mapsto \eta_{j,\kappa}(x/|x|)$ is in $\cG^s(\cA_{R,\ell}: C,C R^{-1}(1+\delta_j^{-1}))$. The constants can be increased to $(C,C2^{-j})$ since $\delta_j  =\pi 2^j/R$  and $2^{j} \leq R$. Because the $\cA_{R,\ell}$ cover $\cA_R$, we deduce that the map $x\to \eta_{j,\kappa}(x/|x|)$ is in $\cG^s(\cA_{R}:C,C2^{-j})$.

On the annulus $\cA_R$, the function $h_{j,\kappa}$ is the product of the functions $x \mapsto \phi_j(|x|)$ and $x \mapsto \eta_{j,\kappa}(x/|x|)$. Thus, by the multiplication law of Gevrey classes (see Lemma \ref{lemma:Gevrey:multiplication}), $h$ is in $\cG^s(\cA_R : A,B 2^{-j})$,  where $A,B$ depend only on $d$,$s$ and the Gevrey constants $A_0,A_1,B_0,B_1$ for the radial and angular cutoffs. Since $h_{j,\kappa}$ is zero outside $\cA_R$, we can extend the same estimate to all of $\R^d$. Therefore, 
$$
|\partial_x^\alpha h_{j,\kappa}(x)|
\leq
A B^{|\alpha|}2^{-j|\alpha|}(\alpha!)^s,
\qquad x\in\R^d .
$$
It remains to consider the top scale $j=j_{\max}$. In this case the cover $\cK_j$ of the sphere is trivial, and so
$
\eta_{j_{\max},\mathbb{S}^{d-1}}\equiv 1$, 
and hence
$$
h_{j_{\max},\mathbb{S}^{d-1}}(x)=
\phi_{j_{\max}}(|x|).
$$
By the construction of the radial partition, we have 
$$
\phi_{j_{\max}}(r)=1
\qquad \text{for }0\leq r\leq 0.45R.
$$
Thus $h_{j_{\max},\mathbb{S}^{d-1}}$ is smooth at the origin. Also, the
derivatives of $\phi_{j_{\max}}(|x|)$ may be nonzero only when
$x$ is in the annulus $\cA_R$. On this annulus, we showed that the map $x\mapsto |x|$ is in $\cG^1(\cA_R:CR,C R^{-1})$. Since $\phi_{j_{\max}}$ is in 
$\cG^s([0,\infty):C,C2^{-j_{\max}})$ and $2^{j_{\max}} = R$,  the composition law implies that $x \mapsto \phi_{j_{\max}}(|x|)$ is in $\cG^s(\cA_R:A,B 2^{-j_{\max}})$ where the constants $A,B$ depend only on $d$ and $s$.

Combining the two cases completes the proof of the lemma.
\end{proof}

Next, we show that the Fourier transform of an interior wave packet $\psi_{j,\kappa, {\bf m}}$ is tightly concentrated about the frequency vector $c_\ast2^{-j} {\bf m}$. In other words, its Fourier transform looks like a bump function around  $c_\ast2^{-j} {\bf m}$.  We prove this in the next lemma. 

\begin{lemma}[Localization of interior wave packets]\label{fourier-decay}
Assume that $0 \leq j \leq j_{\max}$, $\kappa \in \cK_j$, and ${\bf m} \in \mathbb{Z}^d$. Then \[
|\widehat{\psi_{j,\kappa,{\bf m}}}(\xi)|
\le
C\,2^{jd/2}\exp\Bigl(-c\,|2^j\xi-c_\ast{\bf m}|^{1/s}\Bigr),
\qquad \xi \in \mathbb{R}^d, 
\] where  $C,c>0$ are constants depending only on $d$ and $s$.
\end{lemma}
\begin{proof}
Let $h_{j,\kappa}$ be defined by \eqref{eq:def-hjk}, so that
\[
\psi_{j,\kappa,{\bf m}}(x)
=
C_{j,\kappa}2^{-jd/2}h_{j,\kappa}(x)e^{ic_\ast2^{-j}{\bf m}\cdot x}.
\]
By Lemma \ref{derivative--bounds--ball}, $h_{j,\kappa}$ is in $\cG^s(\R^d: A,B 2^{-j})$, where $A, B > 0$ depend only on $d$ and $s$. Observe that $h_{j,\kappa}$ is supported on the sector $S_{j,\kappa}^*$ of Lebesgue measure $\approx  2^{jd}$. By Lemma \ref{lemma:Gevrey:Fourier_decay}, there exist constants $C_1,c_1>0$ depending only on $d$ and $s$ such that
\[
|\widehat{h_{j,\kappa}}(\xi)|
\le
C_1 2^{jd}\exp\Bigl(-c_1|2^j\xi|^{1/s}\Bigr),
\qquad \xi\in\R^d.
\]

Next, using the modulation rule for the Fourier transform,
\begin{align}\label{FT-psi}
\widehat{\psi_{j,\kappa,{\bf m}}}(\xi)
=
C_{j,\kappa}2^{-jd/2}\widehat{h_{j,\kappa}}\bigl(\xi-c_\ast2^{-j}{\bf m}\bigr).
\end{align}
Therefore,
\begin{align*}
|\widehat{\psi_{j,\kappa,{\bf m}}}(\xi)|
&\le
|C_{j,\kappa}|2^{-jd/2}
C_1 2^{jd}
\exp\Bigl(-c_1\bigl|2^j(\xi-c_\ast2^{-j}{\bf m})\bigr|^{1/s}\Bigr) \\
&\leq
C_2 2^{jd/2}\exp\Bigl(-c_1|2^j\xi-c_\ast{\bf m}|^{1/s}\Bigr),
\end{align*}
where we use that $C_{j,\kappa}$ is uniformly bounded by a constant depending on $d$ and $s$.
This proves the lemma.
\end{proof}

\subsubsection{Interior decomposition}

We fix a scale parameter $L > 2\sqrt{d}$ and define  
\[
\cI^{\mathrm{int}}
=\{(j,\kappa,{\bf m}): 0\leq j\leq j_{\max},\ \kappa\in \cK_j,\ {\bf m}\in\mathbb Z^d\} =\cI^{\mathrm{int}}_1\cup \cI^{\mathrm{int}}_2\cup \cI^{\mathrm{int}}_3, 
\]
where 
\begin{align}\notag
\cI^{\mathrm{int}}_1
:=&
\Bigl\{(j,\kappa,{\bf m}):
0\leq j\leq j_{\max},\ \kappa\in \cK_j, \
\dist(c_\ast2^{-j}{\bf m},S)\geq L 2^{-j}
\Bigr\},\\\label{bins}
\cI^{\mathrm{int}}_2
:=&
\Bigl\{(j,\kappa,{\bf m}):
0\leq j\leq j_{\max},\ \kappa\in \cK_j,\ {\bf m}\in\mathbb Z^d,\ 
\dist(c_\ast2^{-j}{\bf m},\R^d\setminus S)\geq L 2^{-j}
\Bigr\},
\\\notag
\cI^{\mathrm{int}}_3
:=&
\Bigl\{(j,\kappa,{\bf m}):
0\leq j\leq j_{\max}, \kappa\in \cK_j, {\bf m}\in\mathbb Z^d, \dist(c_\ast2^{-j}{\bf m},\partial S) \leq L 2^{-j}
\Bigr\}.
\end{align}

We control the cardinality of $\cI_3^{\mathrm{int}}$ in the following lemma.

\begin{lemma}\label{lem:cardinality-residula-int}
For $0 < \eta < 1$,
\[
\#(\cI^{\mathrm{int}}_3) \lesssim_{\eta,d} \cM^{d-\eta}(\partial S) \biggl( R^{d-1}L^{d} + R^{d-\eta}L^{\eta} \biggr),
\] 
while for $\eta = 1$:
\[
\#(\cI^{\mathrm{int}}_3) \lesssim_{d} \cM^{d-1}(\partial S) \biggl( R^{d-1}L^{d} + R^{d-1} \log(R) L \biggr).
\] 
\end{lemma}
\begin{remark}
The estimates used in Lemma~\ref{lem:cardinality-residula-int} are deduced by comparing the number of integer points in the \(cL\)-neighborhood \((\partial S)_{cL}\) of \(\partial{S}\), where \(c\) is some fixed constant and \(L\) is large, to the volume of the neighborhood. For this comparison, we are in a regime where the number of integer points in the neighborhood is genuinely of the size of the volume of the neighborhood.  
Consequently, regardless of how smooth \(\partial{S}\) is, we cannot improve Lemma~\ref{lem:cardinality-residula-int} by using better estimates for the discrepancy (the difference between the volume and the number of integer points inside the neighborhood) such as in \cite{Lettington}.
\end{remark}
\begin{proof}
For $0\leq j\leq j_{\max}$, let
\begin{align*}
M_j &:= \#\Bigl\{{\bf m}\in \mathbb Z^d : \dist(c_\ast2^{-j}{\bf m},\partial S)\leq L 2^{-j} \Bigr\} \\
& \;= \#\Bigl\{{\bf m}\in \mathbb Z^d : \dist({\bf m},c_\ast^{-1}2^j\partial S)\leq L c_\ast^{-1} \Bigr\}.
\end{align*}
Recall $\partial S$ has finite $(d-\eta)$-upper Minkowski content and $1/2 \leq \diam(S) \leq 1$, which implies
\[
| (\partial S)_t | \leq \cM^{d-\eta}(\partial S) t^\eta, \qquad 0 < t \leq 1/2.
\]
Set \(a_j=c_\ast^{-1}2^j\). By scaling the previous estimate,
\[
\left|( \partial (a_j S))_r \right| \leq \cM^{d-\eta}(\partial S) a_j^{d-\eta}r^\eta, \qquad 0<r\leq a_j/2.
\]
(Here, we have used that $(\partial(a_j S))_r = (a_j \partial S)_r = a_j (\partial S)_{r/a_j}$.) For  \(r > a_j\), the estimate $\diam(S) \leq 1$ implies $\diam(\partial( a_j S)) \leq 2 a_j$, which gives the
trivial estimate
\[
\left|( \partial (a_j S))_r \right| \leq C_d r^d, \qquad r > a_j/2.
\]
By setting $r=1/2$ in the definition of Minkowski content and using that $\diam(S) \geq 1/2$, we deduce that 
\[
\cM^{d-\eta}(\partial S) \geq |(\partial S)_{1/2}|/(1/2)^\eta \geq c_d > 0.
\]
Taken together, these estimates imply the volume bound
\[
\left|( \partial (a_j S))_r \right| \leq C_d \cM^{d-\eta}(\partial S) ( a_j^{d-\eta}  r^\eta + r^d ) , \qquad r > 0.
\]
According to \eqref{eqn:vol_bd},
\[
M_j = \# \{ \mathbf{m} \in \Z^d : \dist(\mathbf{m}, \partial(a_j S)) \leq L c_\ast^{-1} \} \leq | (\partial (a_j S))_{L c_\ast^{-1} + \sqrt{d}}| \leq | (\partial (a_j S))_{2L c_\ast^{-1}}|,
\]
where the second inequality follows from $L c_\ast^{-1} > \sqrt{d}$. So, using the volume bound with $r=2L c_\ast^{-1}$,
\[
M_j
\lesssim_{d} \cM^{d-\eta}(\partial S)
(2^{j(d-\eta)}  L^{\eta} + L^{d}).
\]
Thus,
\[
\#(\cI^{\mathrm{int}}_3)
=
\sum_{j=0}^{j_{\max}} \sum_{\kappa\in \cK_j}M_j
=
\sum_{j=0}^{j_{\max}} \#\cK_j\,M_j.
\]
Using the bounds for $\#\cK_j$ and $M_j$, 
\[
\begin{aligned}
\#(\cI^{\mathrm{int}}_3)
&\lesssim_{d} \cM^{d-\eta}(\partial S)
\sum_{j=0}^{j_{\max}}
\left(\frac{R}{2^j}\right)^{d-1}
( 2^{j(d-\eta)}L^{\eta} + L^{d}) \\
&\lesssim_d  
\cM^{d-\eta}(\partial S) \biggl( R^{d-1}L^{\eta}
\sum_{j=0}^{j_{\max}}2^{j(1-\eta)} + R^{d-1}L^{d} \biggr).
\end{aligned}
\]
If \(0<\eta<1\), the sum above is \(O_\eta(R^{1-\eta})\), while for
\(\eta=1\) it is \(O(\log R)\). The estimate stated in the lemma follows immediately.
\end{proof}

\begin{lemma}[Energy estimates I] \label{lem:interior-estimate}
There exist constants $c,C > 0$ determined by $d$ and $s$ such that 
\begin{equation}\label{eqn:energy}
\sum_{\nu\in \cI^{\mathrm{int}}_1}\|\widehat{\psi_\nu}\|_{L^2(S)}^2
+ \sum_{\nu\in \cI^{\mathrm{int}}_2}\|\widehat{\psi_\nu}\|_{L^2(\R^d\setminus S)}^2 \leq C R^d \exp(-cL^{1/s}).
\end{equation}
\end{lemma}
\begin{proof}
For each $0\leq j\leq j_{\max}$,  let
\[
H_j:=\Bigl\{{\bf m}\in\mathbb Z^d:\dist(c_\ast2^{-j}{\bf m},S)\geq L 2^{-j}\Bigr\},
\]
so that
\[
(j,\kappa,{\bf m})\in\mathcal I^{\mathrm{int}}_1
\quad\Longleftrightarrow\quad
0\leq j\leq j_{\max},\ \kappa\in \cK_j,\ {\bf m}\in H_j.
\]
Fix $\nu=(j,\kappa,{\bf m})\in\mathcal I^{\mathrm{int}}_1$.
By Lemma \ref{fourier-decay},
\[
|\widehat{\psi_{j,\kappa,{\bf m}}}(\xi)|^2
\leq C 2^{jd}\exp\Big(-c|2^j\xi-c_\ast{\bf m}|^{1/s}\Big),
\qquad \xi\in\R^d,
\]
for some constants $C, c > 0$ determined by $s$ and $d$. Therefore, 
\begin{align}\label{eqn:energy1}
\sum_{\nu\in \cI^{\mathrm{int}}_1} \|\widehat{\psi_\nu}\|_{L^2(S)}^2 &= \sum_{j=0}^{j_{\max}}\sum_{\kappa\in \cK_j}\sum_{{\bf m}\in H_j}\int_S |\widehat{\psi_{j,\kappa,{\bf m}}}(\xi)|^2 d\xi \\
& \leq C \sum_{j=0}^{j_{\max}}\sum_{\kappa\in \cK_j}2^{jd}\int_S
\sum_{{\bf m}\in H_j}\exp\Bigl(-c|2^j\xi-c_\ast{\bf m}|^{1/s}\Bigr)d\xi. \nonumber
\end{align}
Consider the partition $\{Q_j({\bf m})\}_{{\bf m}\in\mathbb Z^d}$ of $\R^d$ by pairwise disjoint cubes
\[
Q_j({\bf m}) := c_\ast2^{-j}{\bf m}+(0,c_\ast2^{-j}]^d.
\]
For $\xi \in S$, $\mathbf{m} \in H_j$, and $\xi' \in Q_j(\mathbf{m})$, we claim that
\[
\exp\Bigl(-c|2^j\xi-c_\ast{\bf m}|^{1/s}\Bigr)
\le
C\exp\Bigl(-c|2^j\xi-2^j\xi'|^{1/s}\Bigr).
\]
To see this, simply apply the triangle inequality and the concavity of the function $t \mapsto t^{1/s}$ to obtain
\[
|2^j\xi-c_\ast{\bf m}|^{1/s}
\ge
|2^j\xi-2^j\xi'|^{1/s}-|2^j\xi'-c_\ast{\bf m}|^{1/s}
\ge
|2^j\xi-2^j\xi'|^{1/s}-C_0. 
\]
Hence, for fixed $\xi\in S$, by averaging the previous estimate over $\xi' \in Q_j(\mathbf{m})$,
\begin{align}\label{eqn:energy2}
\sum_{{\bf m}\in H_j}\exp\Bigl(-c|2^j\xi-c_\ast{\bf m}|^{1/s}\Bigr)
&\leq C\sum_{{\bf m}\in H_j}\frac{1}{|Q_j({\bf m})|}
\int_{Q_j({\bf m})}\exp\Bigl(-c|2^j(\xi-\xi')|^{1/s}\Bigr)\,d\xi' \\
&\leq C(c_\ast2^{-j})^{-d}\int_{|\xi-\xi'|\geq (L/2)2^{-j}}
\exp\Bigl(-c|2^j(\xi-\xi')|^{1/s}\Bigr)\,d\xi' \nonumber \\
&  =  C c_\ast^{-d}\int_{|\eta|\geq L/2}\exp\Bigl(-2c|\eta|^{1/s}\Bigr)\,d\eta.  \nonumber
\end{align}
In the second line above, we use that for fixed $\xi\in S$ and ${\bf m}\in H_j$ the cube $Q_j({\bf m})$ is contained in $\{\xi'\in\R^d:|\xi-\xi'|\geq (L/2)2^{-j}\}$. Indeed, $\xi' \in Q_j(\mathbf{m}) \implies |\xi' - c_* 2^{-j} \mathbf{m}| \leq c_* \sqrt{d} 2^{-j} \leq \sqrt{d} 2^{-j}$ while $\xi \in S$ and $\mathbf{m} \in H_j \implies |\xi - c_* 2^{-j} \mathbf{m}| \geq L 2^{-j}$. Thus, by the triangle inequality, $|\xi - \xi'| \geq (L - \sqrt{d}) 2^{-j} \geq (L/2) 2^{-j}$.
In the last line above, we use the change of variables $\eta=2^j(\xi-\xi')$. The tail integral satisfies
\begin{equation}\label{eqn:energy3}
\int_{|\eta|\geq L/2}\exp\Bigl(-2c|\eta|^{1/s}\Bigr)\,d\eta
\le
C'\exp\bigl(-c' L^{1/s}\bigr),
\end{equation}
for constants $C',c'>0$ depending only on $d,s$. 

Plugging the lattice sum estimate \eqref{eqn:energy2} and the tail integral bound \eqref{eqn:energy3} into \eqref{eqn:energy1} gives
\[
\sum_{\nu\in \mathcal I^{\mathrm{int}}_1}\|\widehat{\psi_\nu}\|_{L^2(S)}^2
\le
C\sum_{j=0}^{j_{\max}}\sum_{\kappa\in \cK_j}2^{jd}|S| \exp(-c L^{1/s})
\]
Using $|S|\lesssim 1$ (due to the normalization $\diam(S) \leq 1$) and $\#\cK_j\lesssim (R/2^j)^{d-1}$, we obtain
\[
\sum_{\nu\in \mathcal I^{\mathrm{int}}_1}\|\widehat{\psi_\nu}\|_{L^2(S)}^2 \lesssim \sum_{j=0}^{j_{\max}}\Big(\frac{R}{2^j}\Big)^{d-1}2^{jd} \exp(-c L^{1/s}) \lesssim R^{d-1} 2^{j_{\max}} \exp(-c L^{1/s}) = R^d \exp(-c L^{1/s}).
\]
The estimate for the sum over \(\cI_2^{\mathrm{int}}\) is analogous, but one integrates the tail of each packet over
\(\R^d\setminus S\) rather than using the measure of this set. Since the relevant frequency
centers lie in a bounded enlargement of \(S\), the resulting lattice sum is finite and can be bounded by the same approach as above. (The corresponding estimate in the case $d=2$ is treated in \cite{HIM-disk}. See Lemma 4.5 therein.) 
\end{proof}
\subsection{Analysis of boundary wave packets} 
In this subsection, we establish Fourier localization estimates for the boundary wave packets, which are defined in polar coordinates by
\[
\psi_{j,\kappa,{\bf m}}(r,\omega) = C_{j,\kappa} 2^{-j/2} \phi_j(r) \eta_{j,\kappa}(\omega) \exp( \pi i(m_1 2^{-j}r +  {\bf m}' \cdot \Phi_{\kappa^*}(\omega)R/\pi)), \qquad j < 0,
\]
with ${\bf m} = (m_1, {\bf m}') \in \mathbb{Z} \times \mathbb{Z}^{d-1}$ and $\kappa \in \cK_j$;  see \eqref{bdry_wp}.

We now establish the near-exponential decay of the boundary wave packets.

\begin{lemma}[Localization of boundary wave packets]\label{bdrdy-Fourier-decay} 
There exist constants $C,c >0$ determined by $d,s$,  such that for any  $j<0$, $\kappa\in \cK_j$, and ${\bf m}=(m_1, {\bf m}')\in \mathbb Z\times \mathbb Z^{d-1}$, 
\[
|\widehat{\psi_{j,\kappa,{\bf m}}}(\xi)| \leq C 2^{j/2} \exp(-c|m_1|^{1/s}) \exp(-c| {\bf m}'|^{1/s}), \qquad |\xi|\leq 1.
\]
\end{lemma} 
\begin{proof}
Fix $j<0$, $\kappa\in \cK_j$, 
 \({\bf m}=(m_1,{\bf m'})\in\mathbb Z\times\mathbb Z^{d-1}\),  and $\xi\in\R^d$ with $|\xi| \leq 1$. Starting from the formula for $\psi_{j,\kappa,\mathbf{m}}$, and using the  change of variables $(r,\omega) \mapsto (r,y)$ with $y=\Phi_{\kappa^*}(\omega)$, we write
\begin{equation}\label{eqn:radial_FT}
\widehat{\psi}_{j,\kappa,{\bf m}}(\xi) = C_{j,\kappa}\int_{I_j} 2^{-j/2}\phi_j(r)\cA(r) e^{\pi i m_1 2^{-j}r}  dr,
\end{equation}
where
\[
\cA(r) = \int_{B_{d-1}(\pi/R)} \eta_{j,\kappa}(\Phi_{\kappa^*}^{-1}(y)) e^{-i r \Phi_{\kappa^*}^{-1}(y)\cdot \xi} r^{d-1}J_{\kappa^*}(y) e^{i R {\bf m}'\cdot y} dy.
\]
Here, we suppress the dependence of $\cA$ on parameters $j,\kappa,\xi,\mathbf{m}',R$ for notational simplicity. In this expression, $J_{\kappa^*}(y)$ is the Jacobian of the coordinate transformation $\omega = \Phi_{\kappa^*}^{-1}(y)$, so that $d\sigma(\omega)|_{\kappa^*} = J_{\kappa^*}(y) dy$. We are able to restrict this integral to $B_{d-1}(\pi/R)$ because the cutoff $\eta_{j,\kappa}(\omega)$ vanishes on the complement of $\kappa^*$.

Clearly, the map $r\mapsto \cA(r)$ is smooth on $I_j$. We claim that
\begin{equation}\label{eqn:A_gev}
| \partial_r^k \cA(r)| \leq C C^k k! e^{-c |\mathbf{m}'|^{1/s}}, \qquad r \in I_j,
\end{equation}
for constants $C,c > 0$ depending only on $d$ and $s$. In other words, we are claiming that $r \mapsto \cA(r)$ is Gevrey--$1$ with constants $(Ce^{-c|\mathbf{m}'|^{1/s}},C)$. To prove \eqref{eqn:A_gev}, we differentiate under the integral defining $\cA(r)$. Each $r$ derivative either hits the exponential
$e^{-i r \Phi_{\kappa^*}^{-1}(y)\cdot \xi}$, or hits the polynomial $r^{d-1}$. Thus, by the Leibniz rule, $\partial_r^k \cA(r)$ is the sum of $k+1$ terms
\begin{equation}\label{eqn:main_term}
\cA_{k_1,k_2}(r) = \binom{k}{k_1} \int_{B_{d-1}(\pi/R)} \eta_{j,\kappa}(\Phi_{\kappa^*}^{-1}(y)) (-i)^{k_1} (\Phi_{\kappa^*}^{-1}(y)\cdot \xi)^{k_1} e^{-i r \Phi_{\kappa^*}^{-1}(y)\cdot \xi} (\partial_r^{k_2} r^{d-1})J_{\kappa^*}(y) e^{i R {\bf m}'\cdot y} dy,
\end{equation}
over integers $k_1,k_2 \geq 0$ with $k_1 + k_2 = k$.

For fixed $k_1,k_2$, $r\in I_j$ and $|\xi|\leq 1$, we consider the expression
\[
T_{k_1,k_2}(y) = \eta_{j,\kappa}(\Phi_{\kappa^*}^{-1}(y)) (\Phi_{\kappa^*}^{-1}(y)\cdot \xi)^{k_1} e^{-i r \Phi_{\kappa^*}^{-1}(y)\cdot \xi} (\partial_r^{k_2} r^{d-1})J_{\kappa^*}(y), \qquad y\in B_{d-1}(\pi/R).
\]
Up to a constant factor, the integral \eqref{eqn:main_term} is then the value of the Fourier transform of $T_{k_1,k_2}(y)$ at frequency $-R\mathbf{m}'$.

The function $T_{k_1,k_2}(y)$ is the product of four factors that we consider individually. For the first factor, recall condition (A3) of the angular cutoff, which implies that
\[
\| \partial^\alpha_y (\eta_{j,\kappa} \circ \Phi_{\kappa^*}^{-1}) \|_{L^\infty(B_{d-1}(\pi/R))} \leq C_1 C_2^{|\alpha|} (\alpha !)^s (\pi/R)^{-|\alpha|}.
\]
Thus, 
\begin{equation}\label{eqn:factor1}
\eta_{j,\kappa} \circ \Phi_{\kappa^*}^{-1} \in \cG^s(B_{d-1}(\pi/R): C,CR).
\end{equation}
Recall Lemma \ref{lem:cap-chart-jacobian},  which states that the Jacobian $J_{\kappa^*}(y)$ satisfies \begin{equation}\label{eqn:factor2}
J_{\kappa^*} \in \cG^s(B_{d-1}(\pi/R):C,C).
\end{equation}
Also, note that
\begin{equation}\label{eqn:factor3}
|\partial_r^{k_2} r^{d-1}| \leq d^{k_2} r^{d-1-k_2} \leq d^{k_2} R^{d-1-k_2} \leq C^{k} R^{d-1}.
\end{equation}
Finally, we discuss the regularity of the mapping
\[ y\mapsto (\Phi_{\kappa^*}^{-1}(y)\cdot \xi)^{k_1} e^{-ir\Phi_{\kappa^*}^{-1}(y)\cdot \xi} . \]
First, consider the univariate function
\begin{equation}\label{eqn:gev1}
t \in [-1,1] \mapsto t^{k_1} e^{-irt}
\end{equation}
This is a product of the function $t \mapsto e^{-irt}$ in $\cG^1([-1,1]: 1,r)$ and the function $t \mapsto t^{k_1}$ in $\cG^1([-1,1]: k_1!,1)$. Therefore, the function \eqref{eqn:gev1} belongs to $\cG^1([-1,1]:k_1!,2r)$.  Since $r \leq R$, we can say that \eqref{eqn:gev1} belongs to $\cG^1([-1,1]:k_1!,2R)$. Notice that the map $y \in B_{d-1}(\pi/R) \mapsto \Phi_{\kappa^*}^{-1}(y) \cdot \xi$ is real analytic uniformly for all $|\xi| \leq 1$ and $\kappa$, hence this map is in $\cG^1(B_{d-1}(\pi/R) : C,C)$. By the composition law (see Lemma \ref{lemma:Gevrey:composition}), we deduce that
\begin{equation}\label{eqn:factor4}
y \in B_{d-1}(\pi/R) \mapsto (\Phi_{\kappa^*}^{-1}(y) \cdot \xi)^{k_1} e^{-i r \Phi_{\kappa^*}^{-1}(y) \cdot \xi} \mbox{ is in } \cG^1(B_{d-1}(\pi/R) : k_1!, CR).
\end{equation}
By the product property of Gevrey classes (see Lemma \ref{lemma:Gevrey:multiplication}), and by combining the facts \eqref{eqn:factor1}, \eqref{eqn:factor2}, 
\eqref{eqn:factor3} and \eqref{eqn:factor4}, we deduce that
\[
y \mapsto T_{k_1,k_2}(y) \mbox{ is in } \cG^s(B_{d-1}(\pi/R) :  C C^{k} k_1! R^{d-1}, CR)
\]
Furthermore, $T_{k_1,k_2}(y)$ is supported inside $B_{d-1}(\pi/R)$ (because $\eta_{j,\kappa}$ is supported on $\kappa^*$), and after extending by zero it may be regarded as a compactly supported function on \(\mathbb R^{d-1}\). Thus, the measure of the support of $T_{k_1,k_2}(y)$ is less than $CR^{1-d}$. Therefore, Lemma \ref{lemma:Gevrey:Fourier_decay} gives
\[
\left|\cA_{k_1,k_2}(r)\right| =  \binom{k}{k_1} \left|\widehat{T_{k_1,k_2}}(- R{\bf m}')\right|\leq C C^{k} k_1! e^{-c|{\bf m}'|^{1/s}}, \qquad r\in I_j,
\]
where we have bounded the binomial coefficient by $2^k$.

Therefore, since $\partial^k_r \cA(r)$ is a sum of at most $k+1$ terms of the form $\cA_{k_1,k_2}(r)$,  we deduce that
\[
|\partial^k_r \cA(r)| \leq C (k+1) C^{k} k !  e^{-c|{\bf m}'|^{1/s}} \leq C' (C')^k k! e^{-c|{\bf m}'|^{1/s}}.
\]
This completes the proof of \eqref{eqn:A_gev}. Namely, the function $\cA(r)$ is in $\cG^1(I_j : C e^{-c |\mathbf{m'}|^{1/s}},C)$.

Recall properties (R2) and (R3) which state that $\phi_j \in \cG^s(I_j:C,C 2^{-j})$ and $\phi_j$ is supported on $I_j$.

By Lemma \ref{lemma:Gevrey:multiplication}, the product function
\[
b(r):=\phi_j(r)\cA(r)
\]
is in $\cG^s ( I_j: C e^{-c|{\bf m}'|^{1/s}},C 2^{-j})$ and is supported on $I_j$.

Applying Lemma \ref{lemma:Gevrey:Fourier_decay} to the function $b$ at frequency $\pi m_1 2^{-j}$ yields
\[
\left|\int_{I_j} \phi_j(r)\cA(r)e^{\pi i m_1 2^{-j}r} dr\right| \leq  C |I_j| e^{-c|m_1|^{1/s}} e^{-c|{\bf m}'|^{1/s}}.
\]
We return now to \eqref{eqn:radial_FT}. Using $|I_j|\approx 2^j$ which implies $2^{-j/2}|I_j|\approx 2^{j/2}$, and absorbing the uniformly bounded constant $C_{j,\kappa}$ into $C$, we obtain
\[
|\widehat{\psi}_{j,\kappa,{\bf m}}(\xi)| \leq C 2^{j/2} \exp{(-c|m_1|^{1/s})} \exp{(-c|{\bf m}'|^{1/s})}, \qquad |\xi|\leq 1.
\]
This proves the lemma.
\end{proof}

\subsubsection{Boundary decomposition}\label{subsec:boundary-decomposition}
We split the boundary wave packet indices into a high-frequency part and a residual part. Recall
\[
\cI^{\mathrm{bdry}}:=\{(j,\kappa,{\bf m})\in\cI : j<0\}.
\]
Fix two parameters $L_1, L_2 > 1$ and define the residual boundary set
\[
\cI^{\mathrm{bdry}}_3 := \Bigl\{(j,\kappa,{\bf m})\in\cI^{\mathrm{bdry}}: -2L_1 <j<0,\ \kappa \in  \cK_j, \ |m_1|+|{\bf m}'|< L_2 \Bigr\},
\]
and let
\[
\cI^{\mathrm{bdry}}_1:=\cI^{\mathrm{bdry}}\setminus \cI^{\mathrm{bdry}}_3.
\]
Equivalently, $(j,\kappa,{\bf m})\in\cI^{\mathrm{bdry}}_1$ if either $j\leq -2L_1$, or else $-2L_1<j<0$ and $|m_1|+|{\bf m}'|\geq L_2$.

Note that the residual boundary set satisfies
\begin{equation}\label{eqn:bound_card}
\#(\cI^{\mathrm{bdry}}_3)
\lesssim
R^{d-1}L_1 L_2^d. 
\end{equation}
To see this, note that for $j<0$, $\#\cK_j\lesssim R^{d-1}$, the number of integer lattice points $\mathbf{m}$ satisfying $|m_1|+|{\bf m}'|< L_2$ is $\lesssim  L_2^d$, while the number of integers $j$ with $-2L_1<j<0$ is $\leq 2L_1$. Multiplying these three factors gives the claim.

We establish that there exist constants $c,C >0$ determined by $d$ and $s$ so that
\begin{equation}\label{eqn:energy_bdry}
\sum_{\nu\in\cI^{\mathrm{bdry}}_1}\|\widehat{\psi_\nu}\|_{L^2(S)}^2
\leq C R^d(\exp(-cL_1) +  \exp(-c L_2^{1/s})).
\end{equation}
Because $S \subset B_d(1)$, it suffices to bound $\sum_{\nu\in\cI^{\mathrm{bdry}}_1} |\widehat{\psi_\nu}(\xi)|^2$ uniformly in $|\xi| \leq 1$. By Lemma \ref{bdrdy-Fourier-decay},
\[
|\widehat{\psi_{j,\kappa,{\bf m}}}(\xi)|
\leq C 2^{j/2}\exp(-c|m_1|^{1/s})\exp(-c|{\bf m}'|^{1/s}), \qquad|\xi|\leq 1.
\]
If $j\leq -2L_1$, then summing $2^j$ over $j\leq -2L_1$ gives a factor comparable to $2^{-2L_1}= \exp(-cL_1)$.
If $-2L_1<j<0$ and $|m_1|+|{\bf m}'|\geq L_2$, then the exponential tail in ${\bf m}$ gives a factor
$\exp(-cL_2^{1/s})$. The remaining sum over caps contributes at most a factor of $R^{d-1} \leq R^d$. This proves \eqref{eqn:energy_bdry}.

\subsection{Proof of Proposition \ref{prop:energy-estimate}}\label{proof-of-prop:energy-estimate}
Recall that the index set \eqref{eqn:indexset} decomposes as
\[
\cI=\cI^{\mathrm{int}}\cup\cI^{\mathrm{bdry}},
\] and we have the decompositions
\[
\cI^{\mathrm{int}}=\cI^{\mathrm{int}}_1\cup\cI^{\mathrm{int}}_2\cup\cI^{\mathrm{int}}_3,
\qquad
\cI^{\mathrm{bdry}}=\cI^{\mathrm{bdry}}_1\cup\cI^{\mathrm{bdry}}_3.
\]
Here, the index sets are determined by parameters $L, L_1, L_2 > 1$, which we will determine momentarily. Define the global decomposition
\[
\cI_1:=\cI^{\mathrm{int}}_1\cup\cI^{\mathrm{bdry}}_1,
\qquad
\cI_2:=\cI^{\mathrm{int}}_2,
\qquad
\cI_3:=\cI^{\mathrm{int}}_3\cup\cI^{\mathrm{bdry}}_3.
\]
By Lemma \ref{lem:interior-estimate} and the boundary energy estimate \eqref{eqn:energy_bdry},
\begin{equation}\label{eqn:energy_main}
\sum_{\nu\in\cI_1}\|\widehat{\psi_\nu}\|_{L^2(S)}^2 + \sum_{\nu\in\cI_2}\|\widehat{\psi_\nu}\|_{L^2(\R^d\setminus S)}^2
\lesssim R^d( \exp(-cL^{1/s}) + \exp(-cL_1) + \exp(-c L_2^{1/s}))
\end{equation}

According to Lemma \ref{lem:cardinality-residula-int} and the boundary counting estimate \eqref{eqn:bound_card}, if $0 < \eta < 1$ then
\[
\#(\cI_3) = \#(\cI^{\mathrm{int}}_3) + \#(\cI^{\mathrm{bdry}}_3) \lesssim_{d,\eta} 
\cM^{d-\eta}(\partial S) \biggl( R^{d-1}L^{d} + R^{d-\eta}L^{\eta} \biggr) + R^{d-1}L_1 L_2^d.
\]
While for $\eta=1$,
\[
\#(\cI_3) = \#(\cI^{\mathrm{int}}_3) + \#(\cI^{\mathrm{bdry}}_3) \lesssim_d 
 \cM^{d-1}(\partial S) \biggl( R^{d-1}L^{d} + R^{d-1} \log(R) L \biggr) + R^{d-1}L_1 L_2^d.
\]

Set $L= L_2 =  A \log(R/ \varepsilon)^s$ and $L_1 = A \log(R/ \varepsilon)$, where $A$ is a large constant determined by $s,d$. If $A$ is large enough, making these substitions for $L,L_1,L_2$ in \eqref{eqn:energy_main} yields
\[
\sum_{\nu\in\cI_1}\|\widehat{\psi_\nu}\|_{L^2(S)}^2 + \sum_{\nu\in\cI_2}\|\widehat{\psi_\nu}\|_{L^2(\R^d\setminus S)}^2
\leq   \varepsilon^2.
\]
For this choice of $L, L_1, L_2$, we obtain 
\[
\#(\cI_3) \lesssim_{d,s,\eta} 
\left\{
\begin{aligned}
&\cM^{d-\eta}(\partial S) \left(R^{d-1} \log(R/\varepsilon)^{sd} + R^{d-\eta} \log(R/\varepsilon)^{\eta s} \right) + R^{d-1} \log(R/\varepsilon)^{sd+1}, \quad 0 < \eta < 1 \\
&\cM^{d-1}(\partial S) \left(R^{d-1} \log(R/\varepsilon)^{sd} + R^{d-1} \log(R) \log(R/\varepsilon)^{s} \right) + R^{d-1} \log(R/\varepsilon)^{sd+1}, \quad \eta=1
\end{aligned}
\right.
\]
For the case $\eta =1$, observe that the term $R^{d-1} \log (R) \log(R/\varepsilon)^s$ is bounded by $R^{d-1} \log(R/\varepsilon)^{sd}$. In all cases, $\cM^{d-\eta}(\partial S) \gtrsim_d 1$ due to our assumption $\diam(S) \in [1/2,1]$. These comments, taken together, yield the simplification stated in Proposition \ref{prop:energy-estimate}.

\section{Proof of Theorems  \ref{thm:ball-frame-energy} and 
\ref{thm:ball-plunge}}\label{proof-of-main-theorems}\label{sec:proof-of-theorems} 
To prove Theorem \ref{thm:ball-frame-energy} we may reduce to the case $S \subset B_d(1)$, $1/2 \leq \diam(S) \leq 1$. Indeed, by translation and scaling, we may assume that $0\in S$ and $1/2\leq \diam(S)<1$. This is because the spatio-spectral limiting operator associated to $B_d(R)$ and $S$  has the same eigenvalues as the operator associated to $B_d(\Delta R)$, $\Delta^{-1}S$, for any scaling parameter $\Delta$. Moreover, the eigenvalues of limiting operators are preserved under translations of the domains. Then Theorem \ref{thm:ball-frame-energy} is immediate from the frame property established in Lemma \ref{lem:frame}
and the energy concentration estimate in Proposition \ref{prop:energy-estimate}. 
Indeed, Lemma \ref{lem:frame} gives a unit-norm frame $\{\psi_\nu\}_{\nu\in \cI}$ for $L^2(B_d(R))$ with absolute frame bounds.
Proposition \ref{prop:energy-estimate} provides the required decomposition $\cI=\cI_1\cup \cI_2\cup \cI_3$ and the estimates in the statement of
Theorem \ref{thm:ball-frame-energy}. 

For the proof of Theorem \ref{thm:ball-plunge}, we use the following abstract eigenvalue counting lemma for frames
(see \cite{marceca2023}), which extends a lemma from \cite{israel15eigenvalue}.
\begin{lemma}\label{lem:Romero_lem}
Let $T:\mathcal H\to\mathcal H$ be a positive, compact, self-adjoint operator on a Hilbert space $\mathcal H$
with $\|T\|\leq 1$, and let $\{\lambda_n(T)\}_{n\geq 0}$ be its eigenvalues.
Let $\{\phi_i\}_{i\in I}$ be a unit-norm frame for $\mathcal H$ with lower frame bound $A \leq 1$.
If $\cI=\cI_1\cup \cI_2\cup \cI_3$ and
\[
\sum_{i\in \cI_1}\|T\phi_i\|_{\cH}^2+\sum_{i\in \cI_2}\|(I-T)\phi_i\|_{\cH}^2\leq \frac{A}{2}\varepsilon^2,
\]
then
\[
\#\{\lambda_n(T)\in(\varepsilon,1-\varepsilon)\}\leq \frac{2}{A}\,\# \cI_3.
\]
\end{lemma}
\begin{proof}[Proof of Theorem \ref{thm:ball-plunge}]
By applying a rescaling to $S$ and the inverse scaling to $R$, we may assume without loss of generality that $R \geq 2$ is dyadic. Thus, Theorem \ref{thm:ball-frame-energy} applies to produce a unit-norm frame $\{\psi_\nu\}$ for $L^2(B_d(R))$ with frame bounds $0 < A < B < \infty$ determined by $d$, where $A \leq 1 \leq B$. Fix $\varepsilon\in(0,1/2]$ and set
\[
\varepsilon_0=\sqrt{A/2}\,\varepsilon.
\]
Apply Theorem \ref{thm:ball-frame-energy} with $\varepsilon_0$ in place of $\varepsilon$ to obtain a decomposition
$\cI=\cI_1\cup \cI_2\cup \cI_3$ such that
\[
\sum_{\nu\in \cI_1}\|\widehat{\psi_\nu}\|_{L^2(S)}^2
+
\sum_{\nu\in \cI_2}\|\widehat{\psi_\nu}\|_{L^2(\R^d\setminus S)}^2
\leq \varepsilon_0^2,
\]
and 
\[
\# \cI_3 \lesssim_{d,s,\eta}H_\eta(S,R,\varepsilon_0). 
\]
Consider the spatio-spectral limiting operator
\[
T_R=P_{B_d(R)}B_S P_{B_d(R)}:L^2(\R^d)\to L^2(\R^d).
\]
We view each $\psi_\nu\in L^2(B_d(R))$ as an element of $L^2(\R^d)$ by extending it by $0$ outside $B_d(R)$.
Since $P_{B_d(R)}$ is an orthogonal projection, we have
\[
\|T_R\psi_\nu\|_2\leq \|B_S\psi_\nu\|_2,
\qquad
\|(I-T_R)\psi_\nu\|_2\leq \|(I-B_S)\psi_\nu\|_2.
\]
By Plancherel and the definition of $B_S$,
\[
\|B_S\psi_\nu\|_2^2=\|\widehat{\psi_\nu}\|_{L^2(S)}^2,
\qquad
\|(I-B_S)\psi_\nu\|_2^2=\|\widehat{\psi_\nu}\|_{L^2(\R^d\setminus S)}^2.
\]
Therefore,
\[
\sum_{\nu\in \cI_1}\|T_R\psi_\nu\|_2^2+\sum_{\nu\in \cI_2}\|(I-T_R)\psi_\nu\|_2^2
\leq \varepsilon_0^2
=\frac{A}{2}\varepsilon^2.
\]
Applying Lemma \ref{lem:Romero_lem} with $\mathcal H=L^2(B_d(R))$ and $T=T_R$ yields
\[
\#\{k:\lambda_k(T_R)\in(\varepsilon,1-\varepsilon)\}\leq \frac{2}{A}\,\# \cI_3.
\]
Finally, since $\varepsilon_0$ is a fixed constant multiple of $\varepsilon$, we have
$\log(R/\varepsilon_0)\approx \log(R/\varepsilon)$, up to a dimensional constant, and hence $H_\eta(S,R,\varepsilon_0) \approx H_\eta(S,R,\varepsilon)$.
This gives 
\[
\#\{k:\lambda_k(T_R)\in(\varepsilon,1-\varepsilon)\}
\lesssim_{d,s,\eta}H_\eta(S,R,\varepsilon). 
\] 
This completes the proof.
\end{proof}

\section{Proof of Corollary \ref{cor:Landau-quant-esti}}\label{sec:proof-of-corollary}
For the proof of the corollary,  we use the following spectral reduction lemma. 
\begin{lemma}[Spectral reduction lemma]\label{lem:spectral-reduction}
Let \(T\) be a positive trace class contraction (an operator satisfying $0\leq T\leq I$) with eigenvalues
\(1\geq \lambda_1\geq \lambda_2\geq \cdots\ge0\). For
$\varepsilon\in (0,1/2]$, define
\[
N_\varepsilon(T)=\#\{k:\lambda_k>\varepsilon\},
\qquad
P_\varepsilon(T)=\#\{k:\varepsilon<\lambda_k<1-\varepsilon\},
\]
and define the idempotent defect
\[
D(T)=\operatorname{tr}(T-T^2)=\sum_k \lambda_k(1-\lambda_k).
\]
Then
\[
\left|N_\varepsilon(T)-\operatorname{tr}(T)\right|
\le
P_\varepsilon(T)+2 D(T).
\]
\end{lemma}
Therefore, any bound for \(P_\varepsilon(T)\), together with a bound on the idempotent defect $D(T)$,
immediately yields a Weyl-type estimate for \(N_\varepsilon(T)\).
\begin{proof}
Since \(T\) is trace class and \(0\leq T\leq I\), its eigenvalues satisfy
\(\sum_k \lambda_k<\infty\) and \(0\leq \lambda_k\leq 1\). We have
\[
N_\varepsilon(T)-\operatorname{tr}(T)
=
\sum_{\lambda_k>\varepsilon}(1-\lambda_k)
-
\sum_{\lambda_k\leq \varepsilon}\lambda_k .
\]
Hence
\[
N_\varepsilon(T)-\operatorname{tr}(T)
\le
\sum_{\lambda_k>\varepsilon}(1-\lambda_k).
\]
Split the last sum into the plunge region and the near-one region:
\[
\sum_{\lambda_k>\varepsilon}(1-\lambda_k)
\le
P_\varepsilon(T)
+
\sum_{\lambda_k\geq 1-\varepsilon}(1-\lambda_k).
\]
For \(\lambda_k\geq 1-\varepsilon\),
\[
1-\lambda_k\leq (1-\varepsilon)^{-1}\lambda_k(1-\lambda_k),
\]
so
\[
N_\varepsilon(T)-\operatorname{tr}(T)
\le
P_\varepsilon(T)+(1-\varepsilon)^{-1}D(T).
\]
Similarly,
\[
\operatorname{tr}(T)-N_\varepsilon(T)
\le
\sum_{\lambda_k\leq \varepsilon}\lambda_k
\le
(1-\varepsilon)^{-1}\sum_{\lambda_k\leq \varepsilon}\lambda_k(1-\lambda_k)
\le
(1-\varepsilon)^{-1}D(T).
\]
Combining the two estimates and using \(0<\varepsilon \leq 1/2\),  we obtain 
\[
\left|N_\varepsilon(T)-\operatorname{tr}(T)\right|
\le
P_\varepsilon(T)+(1-\varepsilon)^{-1}D(T)
\le
P_\varepsilon(T)+2D(T).
\]
\end{proof} 
\begin{proof}[Proof of Corollary \ref{cor:Landau-quant-esti}] 
 By scaling invariance, we may assume without loss of generality that $1/2 \leq \diam(S) \leq 1$ and  $R \geq 2$ is dyadic. 
 
 First assume
\(0<\varepsilon<1/2\). We are assuming $\partial S$ has finite $(d-\eta)$-upper Minkowski content. By Theorem~\ref{thm:ball-plunge},  
\[P_\varepsilon(T_R)
 \lesssim_{d,s,\eta} H_\eta(S,R,\varepsilon) = 
 \begin{cases}
    \mathcal{M}^{d-\eta}(\partial S)
\left[
R^{d-1}
\log(R/\varepsilon)^{sd+1}
+
R^{d-\eta}
\log(R/\varepsilon)^{s\eta}
\right], &0<\eta<1\\
\mathcal{M}^{d-1}(\partial S) R^{d-1} \log(R/\varepsilon)^{sd+1}, & \eta=1.
\end{cases}
\]

By Theorem~1.4 of  \cite{HIM-trace-rough}, 
together with
Remarks~1.5 and~1.7 there, applied with $F=B_d(1)$ and $\gamma=1$, we obtain
\[
D(T_R)=\operatorname{tr}(T_R-T_R^2)\lesssim_{d,\eta}  \begin{cases}
    \cM^{d-\eta}(\partial S)  R^{d-\eta}, &0<\eta<1\\
  \cM^{d-1}(\partial S) R^{d-1}\log R, & \eta=1.
\end{cases}
\]
The estimates on $D(T_R)$ from \cite{HIM-trace-rough} are stated in terms of the $\eta$-perimeter of $S$, which is controlled by the $(d-\eta)$-upper Minkowski content of $\partial S$. Indeed, Remark~1.7 of \cite{HIM-trace-rough} implies $\mathrm{Per}_\eta(S)\lesssim_d \cM^{d-\eta}(\partial S)$, while Remark~1.5 of \cite{HIM-trace-rough} controls the scale parameters appearing in Theorem~1.4 of \cite{HIM-trace-rough}. Since \(R\ge2\), \(0<\varepsilon<1/2\), and \(sd+1>1\),
\[
\log R\leq \log(R/\varepsilon)
\leq \log(R/\varepsilon)^{sd+1}.
\]
Therefore, 
\[
D(T_R) \lesssim_{d,\eta} H_\eta(S,R,\varepsilon). 
\] 

Since  $T_R$ is   a positive trace-class contraction, we apply Lemma~\ref{lem:spectral-reduction} to obtain 
\[
|N_\varepsilon(T_R)-\operatorname{tr}(T_R)|
\le
P_\varepsilon(T_R)+2D(T_R).
\]
We deduce that
\[
|N_\varepsilon(T_R)-\operatorname{tr}(T_R)|
\lesssim_{d,s,\eta}
H_\eta(S,R,\varepsilon)= H_\eta(S,R,\varepsilon^*).
\]
 
Now assume \(1/2\leq\varepsilon<1\). Set
\[
\tau:=(1-\varepsilon)/2.
\]
Then \(0<\tau\leq 1/4\), and
\[
N_\tau(T_R)-N_\varepsilon(T_R)
=
\#\{k:\tau<\lambda_k(T_R)\leq \varepsilon\}
\le
\#\{k:\tau<\lambda_k(T_R)<1-\tau\}
=
P_\tau(T_R).
\]
Hence,
\[
|N_\varepsilon(T_R)-\operatorname{tr}(T_R)|
\le
|N_\tau(T_R)-\operatorname{tr}(T_R)|+P_\tau(T_R).
\]
Applying  the already proved case to \(\tau\), together with the plunge bound for \(P_\tau(T_R)\), yields
\[
|N_\varepsilon(T_R)-\operatorname{tr}(T_R)|
\lesssim_{d,s,\eta} H_\eta(S,R,\tau).
\]
Since \(\tau=(1-\varepsilon)/2\),  we can check that $H_\eta(S,R,\tau) \lesssim H_\eta(S,R,1-\varepsilon)$. So, we deduce that
\[
|N_\varepsilon(T_R) - \tr(T_R)| \lesssim_{d,s,\eta} H_\eta(S,R,1-\varepsilon) = H_\eta(S,R,\varepsilon^*).
\]
Thus, the desired estimate also holds for \(1/2 \leq \varepsilon < 1\).

Finally, using
\[
\tr(T_R)=(2\pi)^{-d}|B_d(R)||S|
\]
we obtain
\[
\left|
N_\varepsilon(R)-(2\pi)^{-d}|B_d(R)||S|R^d
\right|
\lesssim_{d,s,\eta}
H_\eta(S,R,\varepsilon^*).
\]

\end{proof} 

\printbibliography

\end{document}